\documentclass[a4paper,12pt]{article}

\usepackage[latin1]{inputenc}
\usepackage[T1]{fontenc}
\usepackage[english]{babel}

\usepackage{mathrsfs}
\usepackage{amscd}
\usepackage{amsfonts}
\usepackage{amsmath}
\usepackage{amssymb}
\usepackage{amstext}
\usepackage{amsthm}
\usepackage{amsbsy}

\usepackage{xspace}
\usepackage[all]{xy}
\usepackage{graphicx}
\usepackage{url}
\usepackage{latexsym}

\usepackage{graphicx} 

\usepackage{booktabs} 
\usepackage{array} 
\usepackage{paralist} 
\usepackage{verbatim} 
\usepackage{subfig} 

\usepackage{fancyhdr} 
\usepackage{sectsty}
\allsectionsfont{\sffamily\mdseries\upshape} 

\usepackage[nottoc,notlof,notlot]{tocbibind} 
\usepackage[titles,subfigure]{tocloft} 

\usepackage[textwidth=100pt,textsize=footnotesize,bordercolor=white,color=blue!30]{todonotes}
\usepackage{hyperref} 

\makeatletter
\newcommand*{\rom}[1]{\expandafter\@slowromancap\romannumeral #1@}
\makeatother

\theoremstyle{definition}

\newtheorem{fact}{fact}

\newtheorem{thm}[fact]{Theorem}
\newtheorem{lemma}[fact]{Lemma}
\newtheorem{prop}[fact]{Proposition}
\newtheorem{corollary}[fact]{Corollary}
\newtheorem{defini}[fact]{Definition}

\title{Infinite Time Recognizability from Generic Oracles and the Recognizable Jump Operator}
\author{Merlin Carl}
\date{}

\begin{document}

\maketitle

\begin{abstract}
By a theorem of Sacks, if a real $x$ is recursive relative to all elements of a set of positive Lebesgue measure, $x$ is recursive. This statement, and the analogous statement for non-meagerness instead of positive Lebesgue measure, 
have been shown to carry over to many models of transfinite computations in \cite{CS}. Here, we start exploring another analogue concerning recognizability rather than computability. For a notion of relativized
recognizability (introduced in \cite{Ca1} for ITRMs and generalized here to various other machine types), we show that, for Infinite Time Turing Machines (ITTMs),
if a real $x$ is recognizable relative to all elements of a non-meager Borel set $Y$, then $x$ is recognizable. We also show that a relativized version of this statement holds for Infinite Time Register Machines (ITRMs). 
This extends our work in \cite{Ca3} where we obtained the (unrelativized) result for ITRMs. We then introduce a jump operator for recognizability, 
examine its set-theoretical content and show that the recognizable jumps
for ITRMs and ITTMs are primitive-recursively equivalent, even though these two models are otherwise of vastly different strength. Finally, we introduce degrees of recognizability by considering the transitive closure
of relativized recognizability and connect it with the recognizable jump operator to obtain a solution to Post's problem for degrees of recognizability.

\end{abstract}

\section{Introduction}

It is well-known (see e.g. \cite{BL} or Proposition 2.3 of \cite{Lo}) that, if $x$ is a non-recursive real number, the Turing upper-cone of $x$ is meager. Intuitively, randomly choosing an oracle is not likely to increase the chance of solving
the problem of computing a certain real fixed in advance. 
In a similar spirit, by a theorem of Sacks (see e.g. \cite{DH}), if a real $x$ is recursive relative to all elements of a set $Y$ of positive Lebesgue measure, $x$ is recursive. 
 These statement continue to hold for many machine 
models of infinitary computations as demonstrated in \cite{CS} (for some it is currently still open, while for others it turns out to be
independent of $ZFC$).

Besides computability, there is another way in which an infinitary machine can `determine' a real $x$: $x$ is recognizable if and only if there is a program that halts on all oracles and outputs $1$ when run on the oracle $x$ and otherwise
outputs $0$. Recognizability is known to be a strictly (and in fact much) weaker property than computability. In \cite{Ca1}, a notion of relativized recognizability for ITRMs was considered, resembling computations with oracles.
 This motivates us to ask whether the `random oracles are not informative'-intuition is sufficiently stable to still hold in
this context, i.e. whether recognizability relative to all oracles in some `large' set of reals (i.e. a set of positive Lebesgue measure or a non-meager Borel set) implies recognizability.
In an earlier paper (\cite{Ca3}), we treated the simplest non-trivial case of this question, namely Infinite Time Register Machines (ITRMs) and recognizability from all oracles in a Borel set that is not meager. A strengthening of this result is
proved in section $2$.
A quick inspection of the proof reveals that it makes crucial use of a quite special and convenient property of ITRMs, namely that one can bound the halting time of a program using $n$ registers and running in the oracle $x$
by $\omega_{n+1}^{\text{CK},x}$ from above, an ordinal which in turn has a code computable in the oracle $x$ by an ITRM-program using more registers. As there is, by the work of Hamkins and Lewis (\cite{HL}) a universal Infinite Time Turing Machine (ITTM), we cannot expect this to work for
ITTMs. Hence, we proceed in section $3$ by treating the considerably more delicate case of Infinite Time Turing Machines (ITTMs).

We then strengthen the analogy between computability and recognizability by introducing a jump operator for recognizability and exhibiting some of its basic properties. Finally, we introduce degrees of recognizability and show that,
for ITRMs and ITTMs, there are recognizability degrees strictly between $0$ and its recognizable jump.\\

In particular, we prove the following results: 

\begin{enumerate}
 \item (Theorem \ref{ittrecogcomporstrong}) Assume that $x$ is ITTM-recognizable. Then $x$ is computable or $0^{\prime}_{ITTM}$, the halting real for ITTMs, is ITTM-computable from $x$.
\end{enumerate}
Roughly, this shows that real numbers that are `unique' from the perspective of ITTMs are never incomparable with the halting real for ITTMs.

\begin{enumerate}
\setcounter{enumi}{1}
\item (Theorem \ref{ITTMcomeager}) Let $x$ be uniformly ITTM-recognizable from all elements $y$ of a comeager set $Y$. Then $x$ is ITTM-recognizable.
\end{enumerate}

This is an analogue for ITTM-recognizability of a result by Kleene and Post, according to which the Turing cone above a non-recursive degree is always a meager set (see e.g. Proposition 2.3 of \cite{Lo}). 
By Corollary \ref{relrecogindep}, the analogue for ordinal Turing machines with parameters is independent of ZFC.
(The corresponding statement for ITRMs is the main result of \cite{Ca3}, but also follows from the slightly more general Theorem \ref{ITRMcomeager} below.)

\begin{enumerate}
\setcounter{enumi}{2}

 \item (Theorem \ref{recjumpnonrec}) The recognizable jump $0^{r}$ of $0$ is not recognizable for ITTMs and OTMs without or with a fixed parameter.

 \item (Theorem \ref{intermedrecogdeg}) For $M\in\{\text{ITRM},\text{ITTM}\}$, there is $x$ such that $[0]^{r}_{M}\prec_{M}[x]^{r}_{M}\prec_{M}[0^{r}_{M}]^{r}_{M}$.
\end{enumerate}
This is an analogue of Post's problem for recognizability.

\begin{enumerate}
\setcounter{enumi}{4}
 \item (Theorem \ref{recdeginL}) Let $M\in\{\text{ITRM},\text{ITTM},\text{OTM}\}$. Assume that $V=L$. Then the $M$-recognizability degrees are linearly ordered by the ordering induced by the canonical well-ordering $<_{L}$ of $L$.
\end{enumerate}

Together with Proposition \ref{recdeggenext}, this implies that the degree structure for recognizability is highly dependent on the set-theoretical background. Not even the existence of incomparable degrees (which follows in classical
computability theory from a theorem of Kleene and Post, see e.g. Theorem $1.2$ in chapter VI of \cite{So}) is absolute between transitive class-models of ZFC.

\bigskip

\textbf{Notation}: If $X$ is a set, $\mathfrak{P}(X)$ denotes its power set. 
We fix some natural enumeration $(P_{i}|i\in\omega)$ of the ITRM-programs. For $P$ an ITRM-program, $x\subseteq\omega$ and $i,j\in\omega$, we write $P^{x}(i)\downarrow=j$ for the statement that
$P$, when run in the oracle $x$ with $i$ in its first register and $0$ in all other registers, halts with $j$ in its first register. $P^{x}\downarrow$ abbreviates the statement that the computation of $P$ in the oracle $x$ on the input $0$
halts. The same notation will be used for the other models of computation considered here.

For notions and results on admissible set theory see \cite{Ba} or \cite{Sa}, for descriptive set theory see \cite{Ka}, concerning forcing \cite{Ku}. 
KP denotes Kripke-Platek set theory. $\omega_{i}^{\text{CK},x}$ denotes the $i$th $x$-admissible infinite ordinal, $\omega_{\omega}^{\text{CK},x}:=\text{sup}\{\omega_{i}^{\text{CK},x}:i\in\omega\}$. $\delta$ is the Kronecker symbol, i.e. for
 $x,y\subseteq\omega$, let $\delta(x,y)=1$ if and only if $x=y$ and $\delta(x,y)=0$ otherwise. We say that $A\subseteq[0,1]$ is non-meager if and only if $A$ is Borel and not meager. 
When $(A,\in)$ is a transitive $\in$-structure and $f:\omega\rightarrow A$ is a bijection, then $c:=\{p(i,j):f(i)\in f(j)\}$ is called a code for $(A,\in)$, where $p$ is Cantor's pairing function.

For $M\in\{ITRM, ITTM, OTM\}$, $\text{RECOG}_{M}$ denotes the set of $M$-recognizables (to be defined below).


\section{Infinite Time Register Machines}


Infinite Time Register Machines (ITRMs), introduced in \cite{ITRM} and further studied in \cite{KoMi}, work similarly to the classical unlimited register machines ($URM$s) described in \cite{Cu}. 
In particular, they use finitely many registers each of which can store a single natural number. The difference is that ITRMs use transfinite ordinal running time: The state of an ITRM
at a successor ordinal is obtained as for $URM$s. At limit times, the program line is the inferior limit of the earlier program lines and there is a similar limit rule for the register contents. 
If the inferior limit of the earlier register contents is infinite, the register is reset to $0$.\\

For details on ITRMs, we refer to \cite{Ko}, \cite{KoMi} and \cite{ITRM}. Here, we briefly review some standard notions and facts concerning ITRMs that will be used below.

\begin{defini}
 $x\subseteq\omega$ is ITRM-computable in the oracle $y\subseteq\omega$ if and only if there exists an ITRM-program $P$ such that, for $i\in\omega$, $P$ with oracle $y$ stops for every natural number 
$j$ in its first register at the start of the computation and returns
$1$ if and only if $j\in x$ and otherwise returns $0$. A real ITRM-computable in the empty oracle is simply called ITRM-computable. 
\end{defini}

It is not hard to see that any ITRM-computation either stops or eventually cycles. Moreover, it can be shown (see \cite{KoMi}) that an ITRM-computation eventually cycles if and only if some
state of the computation, consisting of the active program line index $l$ and the register contents $(r_{1},...,r_{n})$ appears at two different times $\alpha<\beta$ such that neither
the active program line index nor any of the register contents drops below their corresponding value at time $\alpha$. This halting criterion can be effectively tested by an ITRM, which leads to the following crucial property
of ITRMs, due to Koepke and Miller:

\begin{thm}{\label{hp}}
Let $\mathbb{P}_{n}$ denote the set of ITRM-programs using at most $n$ registers, and let $(P_{i,n}\mid i\in\omega)$ enumerate $\mathbb{P}_{n}$ in some natural way. 
Then the bounded halting problem $H_{n}^{x}:=\{i\in\omega\mid P_{i,n}^{x}(0)\downarrow\}$ is computable uniformly in the oracle $x$ by an ITRM-program (using more than $n$ registers, of course).\\
\indent
Moroever, if $P\in \mathbb{P}_{n}$, $i\in\omega$, $x\subseteq\omega$ and $P^{x}(i)\downarrow$, then the computation takes less than $\omega_{n+1}^{CK,x}$ many steps.
Consequently, if $P$ is an ITRM-program and $i\in\omega$, $x\subseteq\omega$ are such that $P^{x}(i)\downarrow$, then $P^{x}(i)$ stops in less than $\omega_{\omega}^{CK,x}$ many steps.
\end{thm}
\begin{proof}
 The corresponding results from \cite{KoMi} easily relativize. 
\end{proof}

\begin{thm}{\label{relITRM}}
 Let $x,y\subseteq\omega$. Then $x$ is ITRM-computable in the oracle $y$ if and only if $x\in L_{\omega_{\omega}^{CK,y}}[y]$. Moroever, there is a function $g:\omega\rightarrow\omega$
such that any $x\in L_{\omega_{n}^{CK,y}}[y]$ is computable in the oracle $y$ by some ITRM-program $P$ using at most $g(n)$ registers.
\end{thm}
\begin{proof}
This is a relativization of the main results of Koepke's 
\cite{Ko}.
\end{proof}

\begin{lemma}{\label{ITRMfeatures}}
There are ITRM-programs $(Q_{n}:n\in\omega)$ and $Q$ such that, for every $x\subseteq\omega$:\\
(1) $Q_{n}^{x}$ computes a real number coding $L_{\omega_{n+1}^{CK,x}+2}[x]$.\\
(2) Given a natural number $n$ and a natural number $m$ coding a finite set $p$ of natural numbers, $Q^{x}(m)$ computes a real number $y\supseteq p$ that is Cohen-generic over $L_{\omega_{n+1}^{CK,x}+1}[x]$.
\end{lemma}
\begin{proof}
(1) By standard fine-structural considerations, such a code is contained in $L_{\omega_{n+1}^{CK,x}+3}[x]$ and hence computable by some ITRM-program $P_{x}$ by Theorem \ref{relITRM}. Moroever, there is some $k\in\omega$
such that for each $x$, $P_{x}$ uses at most $k$ many registers. To compute a code for $L_{\omega_{n+1}^{CK,x}+2}[x]$ uniformly in the oracle $x$, we search, starting with $i=0$, through $\omega$ in the following way:
Given $i\in\omega$, first determine, using Theorem \ref{hp}, whether $\forall{j\in\omega}P_{i,k}^{x}(j)\downarrow\{0,1\}$, i.e. whether $P_{i,k}^{x}$ computes a real. If so, determine, using the techniques for evaluating first-order predicates with ITRMs 
from the proof of the lost melody theorem for ITRMs in \cite{ITRM}, whether the real computed by $P_{i,k}^{x}$ is a code for a well-founded $\in$-structure of the form $L_{\alpha}[x]$ such that $\alpha$ is of the form $\beta+2$,
$L_{\beta}[x]\models$KP and $L_{\beta}[x]$ contains exactly $n$ elements of the form $L_{\gamma}[x]$ such that $L_{\gamma}[x]\models KP$. If this holds, then a code as desired has been found; otherwise, proceed with $i+1$.
As we observed, some program in $\mathbb{P}_{k}$ computes a code as desired, so this procedure will terminate for some finite value of $i$.\\

(2) As $L_{\omega_{n+1}^{CK,x}+1}[x]$ is isomorphic (via the Levy collapsing map) to its own $\Sigma_{1}$-Skolem hull of $\{x\}$, it follows that $L_{\omega_{n+1}^{CK,x}+1}[x]$ is countable
in $L_{\omega_{n+1}^{CK,x}+2}[x]$. Hence the proof of the Rasiowa-Sikorski-lemma shows that a real extending $p$ and Cohen-generic over $L_{\omega_{n+1}^{CK,x}+1}[x]$ is contained in $L_{\omega_{n+1}^{CK,x}+2}[x]$.
Use the program $P_{n}$ from (1) to compute a real number $c$ coding $L_{\omega_{n+1}^{CK,x}+2}[x]$. Then search through $\omega$ to determine, again using the techniques for evaluating first-order statements with ITRMs, some $i\in\omega$
 that codes a real with the desired properties in $c$. From $i$ and $c$, the desired real is now easily computable.
\end{proof}

We now define relativized recognizability and then proceed with stating and proving our theorem.


\begin{defini}
Let $x,y\subseteq\omega$. We say that $x$ is ITRM-recognizable from $y$, written $x\leq^{r}_{\text{ITRM}}y$, if and only if there is an ITRM-program $P$ such that $P^{z}\downarrow\in\{0,1\}$ 
for every $z\subseteq\omega$ and, for all $z\subseteq\omega$, we have $P^{z\oplus y}\downarrow=\delta(x,z)$.
For a set $Y\subseteq\mathfrak{P}(\omega)$, we say that $x$ is uniformly recognizable from $Y$ if and only if there is an ITRM-program $P$ such that, for every $y\in Y$ and every $z\subseteq\omega$,
 we have $P^{z\oplus y}\downarrow=\delta(z,x)$. In this case, we say that $x$ is recognized from $Y$ via $P$. We say that $x$ is recognizable if and only if $x\leq^{r}_{\text{ITRM}}0$. 
We denote the set of reals recognizable relative to $y\subseteq\omega$ by RECOG$_{y}$ and abbreviate RECOG$_{0}$ by RECOG.
\end{defini}

\textbf{Remark}: As we are only concerned with ITRMs in this section, we will usually drop the prefix `ITRM' here.\\

\textbf{Remark}: The condition that $P^{z}$ stops with output $0$ or $1$ for every input is introduced merely for the sake of the simplification of further arguments; if $P$ is a program using $n$ registers, we can always use 
the solvability of the bounded halting problem for ITRMs using at most $n$ registers given by Theorem \ref{hp} to produce another program $P^{\prime}$ that, given $z\subseteq\omega$, first tests whether $P^{z}\downarrow$ with output
$0$ or $1$ and returns the output of $P^{z}$ if that is the case and otherwise outputs $0$. $P^{\prime x}$ and $P^{x}$ will hence produce the same output wherever the output of $P$ is of the required form and $P^{\prime}$ will satisfy
our extra condition.\\

A typical phenomenon for models of infinitary computations is the existence of reals that are recognizable, but not computable. As computability is easily seen to imply recognizability, it follows that recognizability is a strictly weaker notion than computability.
 This was first shown in \cite{HL} for Infinite Time Turing Machines.
Detailed treatments of recognizability for ITRMs and for infinitary machines in general can be found in \cite{ITRM}, \cite{Ca}, \cite{Ca1}, \cite{Ca2}.\\

We note here that recognizability is computably stable for ITRMs, i.e. preserved under ITRM-computable equivalence:

\begin{defini}
For $x,y\subseteq\omega$, we write $x\equiv_{ITRM}y$ and say that $x$ and $y$ are ITRM-computably equivalent if and only if there are ITRM-programs $P$ and $Q$ such that $P^{x}$ computes $y$ and $Q^{y}$ computes $x$.
\end{defini}

\begin{prop}{\label{recogstable}}
Let $x\equiv_{ITRM}y$ be real numbers. Then $x\in$ RECOG if and only if $y\in$ RECOG.
\end{prop}
\begin{proof}
Assume that $x\in$ RECOG. Let $P$ and $Q$ be ITRM-programs such that $P^{x}\downarrow=y$ and $Q^{y}\downarrow=x$, and let $R$ be a program for recognizing $x$, i.e. such that
$\forall{z\subseteq\omega}R^{z}\downarrow=\delta(x,z)$. To recognize $y$, we proceed as follows: Assume that $z$ is given in the oracle.\\
\indent
Step $1$: Check, using a halting problem solver (see Theorem \ref{hp}) for $Q$, whether $Q^{z}(i)\downarrow$ for all $i\in\omega$. If not, then $z\neq y$, as $Q$ computes $x$ from $y$ and hence $Q^{y}(i)\downarrow$
for every $i\in\omega$. So in that case, output $0$ and stop. Then check whether $Q^{z}(i)\downarrow\in\{0,1\}$ for all $i\in\omega$ by an exhaustive search. If not, then $z\neq y$, again since $Q^{y}\downarrow=x$, so in that case,
output $0$ and stop. Otherwise, proceed with step $2$.\\
\indent
Step $2$: Let $Q^{z}\downarrow=a$. Check whether $R^{a}\downarrow=1$. If not, then $a\neq x$ as $R$ recognizes $x$, and hence $z\neq y$ as $Q^{y}\downarrow=x$. In that case, output $0$ and stop. Otherwise, proceed with step $3$.\\
\indent
Step $3$: At this point, we know that $Q^{z}\downarrow=a=x$. Check whether $P^{a}\downarrow=z$ (using a halting problem solver as in step $1$). If not, then $z\neq y$ as $P^{x}\downarrow=y$. In this case, output $0$ and stop.
Otherwise, $z=y$, so output $1$ and stop.\\
\indent
Hence $x\in$ RECOG implies $y\in$ RECOG. The reverse direction follows analogously.
\end{proof}

\textbf{Remark}: Note that, however, relativized recognizability is not transitive (see \cite{Ca1}).


\begin{lemma}{\label{genhaltingtimes}}
 Let $P$ be an ITRM-program using $n$ registers, let $x\subseteq\omega$ and suppose that $g$ is Cohen-generic over $L_{\omega_{n+1}^{CK,x}+1}[x]$.
Then $\omega_{i}^{CK,x\oplus g}=\omega_{i}^{CK,x}$ for $i\leq n+1$.
Consequently, $P^{x\oplus g}$ halts in less than $\omega_{n+1}^{CK,x}$ many steps or does not halt at all.
\end{lemma}
\begin{proof}
By Theorem $10.11$ of \cite{Ma}, if $M$ is admissible, $\mathbb{P}$ is a forcing in $M$ and $G$ is $\mathbb{P}$-generic over $M$ such that $G$ intersects
every subclass of $M$ that is a union of a $\Sigma_{1}(M)$ and a $\Pi_{1}(M)$ class, then $M[G]$ is also admissible.
Clearly, as $L_{\alpha+1}[x]$ contains all subclasses of $L_{\alpha}[x]$ definable over $L_{\alpha}[x]$, we have that,
 when $M$ is of the form $L_{\alpha}[x]$ with $x$-admissible $\alpha$ and $g\subseteq\omega$ Cohen-generic over $L_{\alpha+1}[x]$, then
$L_{\alpha}[x][g]$ is admissible.\\
\indent
As admissible ordinals are indecomposable, it follows from Theorem $9.0$ of \cite{Ma} that the forcing extension $L_{\omega_{i}^{CK,x}}[x][g]$ agrees with
the relativized $L$-level $L_{\omega_{i}^{CK,x}}[x\oplus g]$ for all $i\leq n+1$.
Consequently, if $g$ is as in the assumption of the lemma, then $L_{\omega_{i}^{CK,x}}[x][g]=L_{\omega_{i}^{CK,x}}[x\oplus g]$ is admissible for $i\leq n+1$:
so $\omega_{i}^{CK,x}$ is $x\oplus g$-admissible for $i\leq n+1$. Certainly, every $x\oplus g$-admissible ordinal is also $x$-admissible, so that first $(n+1)$ many
$x$-admissible ordinals agree with the first $(n+1)$ many $x\oplus g$-admissible ordinals. Hence $\omega_{n+1}^{CK,x}=\omega_{n+1}^{CK,x\oplus g}$.\\
\indent
The second claim now follows from Theorem \ref{hp}.
\end{proof}

In \cite{Ca3}, we showed that every real $x$ that is uniformly recognizable from all elements of a non-meager Borel set is recognizable. Here, we demonstrate a relativized version of this result (though the proof pretty much remains the same).

\begin{defini}{\label{extr}}
A real number $z$ is an $r$-extracting real if and only if there are a comeager set of reals $Y$ and a real number $x$ such that $x\not\leq^{r}_{\text{ITRM}}z$, but $x$ is uniformly recognizable from $\{z\}\oplus Y:=\{z\oplus y:y\in Y\}$. 

\end{defini}

Intuitively, a real $x$ is $r$-extracting when addition of a typical oracle to $x$ allows to extract more information than one gets from $x$ itself. We will now show that $r$-extracting reals do not exist; the statement
that recognizability from all oracles in a comeager set implies recognizability is in this language the special case that $0$ is not $r$-extracting. This special case was proved as Theorem $7$ of \cite{Ca3}.

\begin{thm}{\label{ITRMcomeager}}
There are no uniformly $r$-extracting reals. I.e.: If $z\subseteq\omega$, $Y\subseteq[0,1]$ is comeager, and $x\subseteq\omega$ is uniformly recognizable in $\{z\}\oplus Y$, then $x$ is recognizable from $z$.
\end{thm}
\begin{proof}
Suppose that $Y$ is comeager, $z\subseteq\omega$ and $x\subseteq\omega$ is uniformly recognized from all elements of $\{z\}\oplus Y$ by the ITRM-program $P$. Assume that $P$ uses $n$ registers.
Let $C$ be the set of real numbers that are Cohen-generic over $L_{\omega_{n+1}^{\text{CK},x\oplus z}+1}[x\oplus z]$. As $C$ is comeager, so is $Y\cap C$. We can hence assume without loss 
of generality that $Y\subseteq C$. Pick $y\in Y$. By assumption, we have $P^{x\oplus (z\oplus y)}\downarrow=1$. By Lemma \ref{genhaltingtimes}, we have
$\omega_{n+1}^{\text{CK},x\oplus (z\oplus y)}=\omega_{n+1}^{\text{CK},(x\oplus z)\oplus y}=\omega_{n+1}^{\text{CK},x\oplus z}$, so $P^{x\oplus (z\oplus y)}$ runs for $<\omega_{n+1}^{\text{CK},x\oplus z}$ many steps
and hence halts inside $L_{\omega_{n+1}^{\text{CK},x\oplus z}}[x\oplus (z\oplus y)]$. By the forcing theorem for admissible sets (see Lemma 10.10 of \cite{Ma}), there must be some forcing condition $p\subseteq y$
such that $p\Vdash P^{x\oplus (z\oplus y)}\downarrow=1$ over $L_{\omega_{n+1}^{\text{CK},x\oplus z}}[x\oplus z]$. Hence, we have $P^{x\oplus(z\oplus g)}\downarrow=1$ for every $g\supseteq p$ that is Cohen-generic over 
$L_{\omega_{n+1}^{\text{CK},x\oplus z}+1}[x\oplus z]$.

We claim that the following procedure recognizes $x$ relative to $z$: Given a real $a$ in the oracle, compute (a code for) $L_{\omega_{n+1}^{\text{CK},a\oplus z}+1}[a\oplus z]$, using (1) of Lemma \ref{ITRMfeatures}.
From that code, compute a real $g_{a}\supseteq p$ that is Cohen-generic over $L_{\omega_{n+1}^{\text{CK},a\oplus z}+1}[a\oplus z]$, using (2) of Lemma \ref{ITRMfeatures}. Then run $P^{a\oplus(z\oplus g_{a})}$,
which must, by assumption, halt with output $0$ or $1$, and return its output. Let $Q$ be an ITRM-program that carries out this procedure when given the oracle $a\oplus z$.

We need to see that $Q^{a\oplus z}\downarrow=1$ if and only if $a=x$, and otherwise $Q^{a\oplus z}\downarrow=0$. That $Q^{x\oplus z}\downarrow=1$ is clear since the generic $g$ picked in the procedure extends $p$,
which forces $Q^{x\oplus (z\oplus g)}$ to converge to $1$ in $L_{\omega_{n+1}^{\text{CK},x\oplus z}}[x\oplus(z\oplus g)]$ and the computation is absolute between $L_{\omega_{n+1}^{\text{CK},x\oplus z}}[x\oplus(z\oplus g)]$ and
the real world.

Now suppose that $a\neq x$, but that $Q^{a\oplus z}\downarrow=1$. This implies that, for some $g_{a}$ Cohen-generic over $L_{\omega_{n+1}^{\text{CK},a\oplus z}+1}[a\oplus z]$,
we have $P^{a\oplus(z\oplus g_{a})}\downarrow=1$. By the forcing theorem for admissible sets again, there is some forcing condition $q\subseteq g_{a}$ such that
$q\Vdash P^{a\oplus(z\oplus g_{a})}\downarrow=1$. Consequently, $P^{a\oplus(z\oplus \bar{g})}\downarrow=1$ holds for every $\bar{g}\supseteq q$ that is Cohen-generic over
$L_{\omega_{n+1}^{\text{CK},a\oplus z}+1}[a\oplus z]$. But the set $\hat{C}$ of all such $\bar{g}$ is not meager and must hence intersect $Y$; so let $g\in Y\cap\hat{C}$.
Then $P^{a\oplus(z\oplus g)}\downarrow=1$, but $a\neq x$, which contradicts the assumption that $P$ recognizes $x$ from $\{z\}\oplus Y$.
\end{proof}

We can relax the condition of $Y$ being comeager to $Y$ merely being Borel and not meager.

\begin{corollary}{\label{ITRMnonmeager}}
Let $Y\subseteq[0,1]$ be non-meager, $z\subseteq\omega$, and let $x\subseteq\omega$ be uniformly recognizable from $\{z\}\oplus Y$. Then $x\leq^{r}_{\text{ITRM}}z$.
\end{corollary}
\begin{proof}
As $Y$ is non-meager, there is an interval $I=(a,b)\subseteq[0,1]$ such that $Y$ is comeager in $I$. By shortening $I$ if necessary, we may assume without loss of generality that $I$ 
is of the form $\{tx|x\in^{\omega}2\}$ for $t\in^{<\omega}2$ (where $tx$ denotes the concatenation of $t$ and $x$) and (passing to $Y\cap I$ if necessary) that $Y\subseteq I$.
 Suppose that $P$ recognizes $x$ relative to all elements of $\{z\}\oplus Y$. We define a program $P^{\prime}$ that recognizes $x$ from all elements of $\{z\}\oplus Y^{\prime}$ where $Y^{\prime}:=\{x:tx\in I\}$, which is obviously a comeager set.
$P^{\prime a\oplus (z\oplus y)}$ works by simply running $P^{a\oplus (z\oplus ty)}$. Clearly, $P^{\prime}$ has the desired properties: For $y\in Y^{\prime}$, we have $P^{\prime a\oplus (z\oplus y)}\downarrow=1$ if and only if $P^{a\oplus (z\oplus ty)}\downarrow=1$
which, as $ty\in Y$ by definition of $Y^{\prime}$, is equivalent with $a=x$. So $P^{\prime}$ recognizes $x$ from all elements of $\{z\oplus Y\}$ for a comeager set $Y$. By Theorem \ref{ITRMcomeager}, $x$ is then uniformly recognizable from $z$.
\end{proof}

We have so far worked with uniform recognizability, i.e. the program recognizing $x$ from $z$ and some $y\in Y$ has to be the same for all elements of $Y$. If one wants to drop this assumption and allow $x$
to be recognized from $y$ by different programs $P$ for different $y\in Y$, the problem arises that the corresponding subsets $Y_{P}^{z}:=\{y\in Y:P$ recognizes $x$ from $z\oplus y\}$ might not have the property
of Baire and hence not be comeager in some interval so that Theorem \ref{ITRMcomeager} is not applicable. At least under some (standard) set-theoretical extra assumption, however, we can strengthen the claim to
drop the uniformity condition:

\begin{corollary}{\label{MAnonuniform}}
Assume that every $\bf\Sigma^{1}_{2}$-set of reals has the Baire property. Let $Y$ be a non-meager set, $x,z\subseteq\omega$ and assume that, for every $y\in Y$, there is some ITRM-program $P_{y}$
such that for all $a\subseteq\omega$, $P_{y}^{a\oplus(z\oplus y)}\downarrow=1$ if and only if $a=x$ and otherwise $P_{y}^{a\oplus(z\oplus y)}\downarrow=0$. Then $x$ is ITRM-recognizable from $z$.
\end{corollary}
\begin{proof}
$Y_{P}^{z}$ is $\Pi^{1}_{2}$ in $x$ and $z$ for every ITRM-program $P$: Namely, the set of these $y$ is definable by a formula $\phi$
expressing `For all $a,b\subseteq\omega$: If $b$ codes the computation of $P$ in the oracle $a\oplus (z\oplus y)$ and this computation stops with output $1$, then $a=x$'. 
(Recall that, by the choice of $P$, the computation of $P$ any oracle always terminates
and hence is a countable set codable by a real.) As `$b$ codes the computation of $P$ in the oracle $c$' is $\Pi_{1}^{1}$ in $c$, the negation is $\Sigma_{1}^{1}$, so the statement
`For all $a,b\subseteq\omega$: $a=x$ or $b$ does not code the computation of $P$ in the oracle $a\oplus (z\oplus y)$', which is equivalent to $\phi$, is $\Pi^{1}_{2}$ in $z$ and $x$.

Thus, for each ITRM-program $P$, $Y_{P}$ has the Baire property (as its complement is $\bf\Sigma^{1}_{2}$ and hence Baire by assumption
and as complements of Baire sets are again Baire). Now $Y=\bigcup_{i\in\omega}Y_{P_{i}}$. As $Y$ is not meager,
it cannot be the union of countably many meager sets. So there is some $k\in\omega$ such that $\bar{Y}:=Y_{P_{k}}$ is not meager. As $\bar{Y}$ also has the Baire property, there is an interval such that $\bar{Y}$ is
comeager relative to that interval. As in the proof of Corollary \ref{ITRMnonmeager}, it follows that $x$ is uniformly ITRM-recognizable from $z\oplus y$ for all elements $y$ of a comeager set of oracles, and hence, by Theorem \ref{ITRMcomeager},
$x$ is ITRM-recognizable from $z$.
\end{proof}

\textbf{Remark}: The assumption that every $\bf\Sigma^{1}_{2}$-set of reals has the Baire property follows for example from the existence of a measurable cardinal (see e.g. Corollary $14.3$ \cite{Ka}).
By Proposition $13.7$ of \cite{Ka}, every $\bf\Sigma_{2}^{1}$-set is a union of $\aleph_{1}$ many Borel sets. By Theorem $2.20$ of chapter II of \cite{Ku}, MA$_{\omega_{1}}$ implies that a union of $\aleph_{1}$ many
meager sets is meager (and hence that a union of $\aleph_{1}$ many sets with the Baire property has the Baire property). As Borel sets have the Baire property, it thus also follows from MA$_{\omega_{1}}$ that
all $\bf\Sigma^{1}_{2}$-sets have the Baire property.

Moreover, the statement that all $\bf\Sigma^{1}_{2}$-sets of reals have the Baire property is equivalent to the statement that the set of reals that are Cohen-generic over $L[x]$ is comeager for every $x\subseteq\omega$ (see Theorem $14.2$ of 
\cite{Ka}).

It well known that $L$ contains $\bf\Sigma^{1}_{2}$-sets of reals that fail to have the Baire property (see e.g. Corollary $13.10$ of \cite{Ka} or observe that in $L$, the set of reals Cohen-generic over $L$ is empty).
 On the other hand, MA$_{\omega_{1}}$ holds in a forcing extension of $L$ (see Theorem $10.11$ of \cite{Ka}). Thus, the statement
that every $\bf\Sigma^{1}_{2}$-set of reals has the Baire property is independent of $ZFC$.


\indent

\textbf{Question}: Is this non-uniform version of Corollary \ref{ITRMnonmeager} is provable in $ZFC$ alone?

\section{Infinite Time Turing Machines}

Since there is in the case of ITTMs no analogue for the stratification of halting times as given by Theorem \ref{hp} for ITRMs, which is a crucial ingredient of the proof of Theorem \ref{ITRMcomeager}, the treatment of ITTMs
will require a new idea. This new idea is that, for a certain $x$ recognizable from many oracles, there is a large subset of the oracles for which the full strength of $\lambda^{x}$ (the supremum of the ITTM-halting times in the oracle $x$) is unnecessary
for performing the recognition: We can limit the `relevant' halting times to a certain ordinal $\alpha<\lambda^{x}$. This serves a similar purpose as the stratification of halting times for ITRMs.


\begin{defini}{\label{ITTMdef}}
Let $x\subseteq\omega$. $\lambda^{x}$ denotes the supremum of ITTM-halting times in the oracle $x$. A real $x$ is ITTM-computable (or ITTM-writable) in the oracle $y$ if and only if there is an ITTM-program $P$ such that $P^{y}$
halts with $x$ on the output tape. A real $x$ is eventually writable in the oracle $y$ if and only if there is an ITTM-program $P$ such that $P^{y}$ does not halt, but eventually leaves the content of the output tape invariant and equal to $x$.
$\zeta^{x}$ denotes the supremum of those ordinals $\alpha$ such that $L_{\alpha+1}[x]\setminus L_{\alpha}[x]$ contains a real that is eventually writable in the oracle $x$.
A real number $x$ is accidentally writable in the oracle $y$ if and only if there is an ITTM-program $P$ such that the tape content of the computation of $P^{y}$ is equal to $x$ at some point of time.
$\Sigma^{x}$ denotes the supremum of those ordinals $\alpha$ such that $L_{\alpha+1}[x]\setminus L_{\alpha}[x]$ contains a real that is accidentally writable in $x$.
\end{defini}

\begin{thm}{\label{ITTMchar}}(Welch)
A real $y$ is ITTM-computable in the oracle $x$ if and only if $y\in L_{\lambda^{x}}[x]$, eventually writable if and only if $y\in L_{\zeta^{x}}[x]$ and accidentally writable if and only if $y\in L_{\Sigma^{x}}[x]$.
Furthermore, $(\lambda^{x},\zeta^{x},\Sigma^{x})$ is the lexically minimal triple $(\alpha,\beta,\gamma)$ of ordinals such that $L_{\alpha}[x]\prec_{\Sigma_{1}}L_{\beta}[x]\prec_{\Sigma_{2}}L_{\gamma}[x]$.
\end{thm}
\begin{proof}
 See \cite{We3}.
\end{proof}

\begin{thm}{\label{lambdalimit}}(Hamkins-Lewis)
For each real $x$, $\lambda^{x}$ is $x$-admissible, a limit of $x$-admissible ordinals and a limit of $x$-admissible limits of $x$-admissible ordinals.
\end{thm}
\begin{proof}
This follows from the `Indescribability Theorem' $8.3$ of \cite{HL} (which in fact shows that we could iterate the `limit of'-operation as often as we wanted) whose proof easily relativizes.
\end{proof}

\begin{lemma}{\label{retractrecog}}
Let $x$ be ITTM-recognizable by the program $P$. Then $x\in L_{\lambda^{x}}$ and $x$ is the unique witness to some $\Sigma_{1}$-formula $\phi$ in $L_{\lambda^{x}}$.
\end{lemma}
\begin{proof}
By definition of $\lambda^{x}$, $P^{x}$ halts in less than $\lambda^{x}$ many steps, hence $L_{\lambda^{x}}[x]\models P^{x}\downarrow=1$. Furthermore, as $P$ recognizes $x$,
$x$ is the unique element of $L_{\lambda^{x}}[x]$ with this property. As $P^{y}\downarrow=1$ is $\Sigma_{1}$-expressible in the parameter $y$, so is $\exists{y}P^{y}\downarrow=1$.
Now, $\lambda^{x}$ is a limit of admissible ordinals by Theorem \ref{lambdalimit}. 

 By a theorem of Jensen and Karp (\cite{JK}), set theoretical $\Sigma_{1}$-formulas are absolute between 
$L_{\alpha}$ and $V_{\alpha}$ when $\alpha$ is a limit of admissible ordinals. As $x\in V_{\lambda^{x}}$, we have $V_{\lambda^{x}}\models\exists{y}P^{y}\downarrow=1$. Hence $L_{\lambda^{x}}\models\exists{y}P^{y}\downarrow=1$.
So there is $y\in L_{\lambda^{x}}$ such that $L_{\lambda^{x}}\models P^{y}\downarrow=1$. By absoluteness of computations, $P^{y}\downarrow=1$ holds in the real world.
As $P$ recognizes $x$, we must have $y=x$. So $x\in L_{\lambda^{x}}$, and $x$ is the unique witness in $L_{\lambda^{x}}$ to the $\Sigma_{1}$-formula $\exists{y}P^{y}\downarrow=1$.
\end{proof}

We note that, as a consequence, there are no recognizable intermediate degrees for ITTMs. An important question in classical recursion theory is whether there exist `natural', `specific' examples of reals
incomparable with $0^{\prime}$ in the sense of Turing reducibility (see e.g. \cite{Si}). For infinitary machines, it seems sensible to understand `specific' as `recognizable'. The following can hence be seen
as a proof that such reals do not exist for ITTMs. A similar result for ITRMs was obtained in \cite{Ca4}.

\begin{thm}{\label{ittrecogcomporstrong}}
Assume that $x$ is ITTM-recognizable. Then $x$ is computable or $0^{\prime}_{ITTM}$, the halting real for ITTMs, is ITTM-computable from $x$.
\end{thm}
\begin{proof}
Assume that $x$ is a lost melody, i.e. recognizable, but not computable. So, by Lemma \ref{retractrecog}, we have $x\in L_{\lambda^{x}}-L_{\lambda}$.
So $\lambda^{x}>\lambda$. If $P$ is a halting ITTM-program, then $P$ will also halt inside the $\Sigma_{1}$-Skolem hull $\mathcal{H}$ of $\emptyset$ in $L_{\lambda}$,
so $\mathcal{H}$ must contain the computation of $P$ for each halting $P$. 
Thus, $\mathcal{H}$ is isomorphic to $L_{\lambda}$, so that a surjection of $\omega$ onto $L_{\lambda}$
is definable over $L_{\lambda}$. Hence $\lambda$ is an index, i.e. $(L_{\lambda+1}-L_{\lambda})\cap\mathbb{R}\neq\emptyset$. By a theorem of Boolos and Putnam (see \cite{BP}),
$L_{\lambda+1}$ contains an arithmetical copy of $L_{\lambda}$, i.e. $\text{cc}(L_{\lambda})\in L_{\lambda+1}$. Hence $\text{cc}(L_{\lambda})\in L_{\lambda^{x}}[x]$,
so $\text{cc}(L_{\lambda})$ is computable from $x$. However, from $\text{cc}(L_{\lambda})$, it is easy to obtain $0^{\prime}_{ITTM}$ by using $\text{cc}(L_{\lambda})$
for evaluating statements of the form $P_{i}\downarrow$ in $L_{\lambda}$.
\end{proof}


\begin{lemma}{\label{genericlambda}}
Let $x\subseteq\omega$, and let $y$ be Cohen-generic over $L_{\Sigma^{x}+1}[x]$. Then $\lambda^{x\oplus y}=\lambda^{x}$.
\end{lemma}
\begin{proof}
This is the relativized version of a fact used in the proof of Theorem $3.1$ of \cite{We}.
The proof given in Lemma $33$ of \cite{CS} relativizes.
\end{proof}

%

We will need the following Theorem due to A. Mathias, which implies that, under rather mild assumptions on the closedness of $L_{\alpha}[x]$, the forcing extensions some of some $L_{\alpha}[x]$ by a generic
real $y$ is the same as $L_{\alpha}[x\oplus y]$. Here, $P_{\theta}^{e}$ denotes the provident hierachy relativized to $e$ and $(P_{\theta}^{e})^{\mathbb{P}}[G]$ denotes the generic extension of $P_{\theta}^{e}$ by $G$, where
$G$ is a $\mathbb{P}$-generic filter over $P_{\theta}^{e}$. For the definition of the provident hierarchy and of a provident set, see \cite{Ma}. 

\begin{thm}{\label{math}}
 Let $\theta$ be an indecomposable ordinal strictly greater than the rank of a transitive set $e$ which contains the notion of forcing $\mathbb{P}$.
Let $G$ be $(P_{\theta}^{e},\mathbb{P})$-generic. Then $(P_{\theta}^{e})^{\mathbb{P}}[G]=P_{\theta}^{e\cup\{G\}}$.
\end{thm}
\begin{proof}
 See \cite{Ma}, Theorem $9$. 
\end{proof}

\textbf{Remark}: This theorem does not immediately apply in the case of real numbers as these are in general not transitive; we can e.g. circumvent this problem by replacing $e\subseteq\omega$ by the transitive set $e^{\prime}:=\omega\cup\{\{i+1\}:i\in e\}$.
As we will only use this theorem in the context of Cohen-forcing which is contained in $L_{\omega+6}$, the condition that $e$ must contain the notion of forcing is not relevant for our purposes: The relevant information
can always be encoded in a transitive set of rank $\omega+i$ where $i\in\omega$.

Moreoever, if $x$ is a real number and $\alpha$ is $x$-admissible, then we have $P_{\alpha}[x]=J_{\alpha}[x]=L_{\alpha}[x]$. As these hierarchies are continuous by definition, the same holds when $\alpha$
is a limit of $x$-admissible ordinals. 

\begin{corollary}{\label{relforc}}
Let $x,y\subseteq\omega$. Write $L_{\alpha}^{x}$ for the $\alpha$-th level of the $L$-hierarchy relativized to $x$ and, if $y$ is generic over some $M$, write $M[y]$ for the generic extension. Suppose that
 $\alpha$ is $x$-admissible or a limit of $x$-admissible ordinals and that $y$ is Cohen-generic over $L_{\alpha}[x]$. Then $L_{\alpha}^{x\oplus y}=L_{\alpha}^{x}[y]$.
\end{corollary}

From now on, we therefore can and will use the square bracket notation in both cases.


\begin{lemma}{\label{relJK}}
Let $x\subseteq\omega$, $\phi$ a $\Sigma_{1}$-formula in the parameter $x$, possibly with real parameters in $L_{\alpha}[x]$ where $\alpha$ is a limit of $x$-admissible ordinals.
Then $\phi$ is absolute between $V_{\alpha}$ and $L_{\alpha}[x]$.
\end{lemma}
\begin{proof}
This is a relativization of the Jensen-Karp theorem in section $5$ of \cite{JK}. The proof more or less relativizes, we elaborate on the necessary changes in the appendix, see Theorem \ref{parameterJK}. 
\end{proof}

\begin{lemma}{\label{ITTMrelrecog}}
Assume that a real $x$ is recognizable relative to a real $y$. Then $x\in L_{\lambda^{x\oplus y}}[y]$.
\end{lemma}
\begin{proof}
Here, we use Lemma \ref{relJK}: Since $\lambda^{x\oplus y}$ is a limit of $x\oplus y$-admissible ordinals and each $x\oplus y$-admissible ordinal is in particular $y$-admissible, $\lambda^{x\oplus y}$ is a limit of $y$-admissible ordinals.
As $\exists{z}P^{z\oplus y}\downarrow=1$ is $\Sigma_{1}$, it is hence absolute between $V_{\lambda^{x\oplus y}}$ and $L_{\lambda^{x\oplus y}}[y]$, as
the latter contains the necessary parameter $y$ (as $\lambda^{x\oplus y}$ is a limit of $x\oplus y$-admissibles by Theorem \ref{lambdalimit} and hence of $y$-admissibles). Consequently, it holds in $L_{\lambda^{x\oplus y}}[y]$, so this structure
contains $x$ by absoluteness of computations.
\end{proof}

\begin{lemma}{\label{ITTMmanyoraclesinL}}
Assume that a real $x$ is ITTM-recognizable relative to all $y\in M$, where $M\subseteq\mathfrak{P}(\omega)$ is Borel and non-meager. Then $x\in L_{\lambda^{x}}$.
\end{lemma}
\begin{proof}
The set will contain mutually generics $g_{1},g_{2}$ over $L_{\Sigma^{x}+1}[x]$ by Lemma 30 of \cite{CS}. We have $\lambda^{x\oplus g_{1}}=\lambda^{x\oplus g_{2}}=\lambda^{x}$ by 
Lemma \ref{genericlambda}. So, by Lemma \ref{ITTMrelrecog}, Lemma \ref{relforc} and Lemma $28$ of \cite{CS}, 
we will have $x\in L_{\lambda^{x\oplus g_{1}}}[x\oplus g_{1}]\cap L_{\lambda^{x\oplus g_{2}}}[x\oplus g_{2}]=L_{\lambda^{x}}[g_{1}]\cap L_{\lambda^{x}}[g_{2}]=L_{\lambda^{x}}$, as desired.
\end{proof}


\begin{thm}{\label{ITTMcomeager}}
Let $x$ be uniformly ITTM-recognizable from all elements $y$ of a comeager set $Y$. Then $x$ is ITTM-recognizable.
\end{thm}
\begin{proof}
 Let $P$ be an ITTM-program that recognizes $x$ relative to every element $y\in Y$. By Lemma \ref{ITTMmanyoraclesinL}, we have $x\in L_{\lambda^{x}}$. The set $C_{\beta}^{z}$ of reals Cohen-generic over $L_{\beta}[z]$ is comeager for every countable ordinal $\beta$
and every real $z$. We may hence assume without loss of generality that $Y\subseteq C_{\Sigma^{x}+1}^{x}$. By Lemma \ref{genericlambda} then, $\lambda^{x\oplus y}=\lambda^{x}$ holds for all $y\in Y$.
Hence $P^{x\oplus y}\downarrow=1$ in less than $\lambda^{x}$ many steps for every $y\in Y$. For $y\in Y$, let $\tau(y)$ be the halting time of $P^{x\oplus y}$. Then $\tau$ (as a function from $\mathbb{R}$ to $\lambda^{x}$)
has comeager pre-image and countable domain, hence there is some
$\zeta<\lambda^{x}$ such that $\tau^{-1}[\zeta]$ is not meager (since otherwise, $\mathbb{R}$ was a countable union of meager sets, i.e. meager). Let $\zeta<\lambda^{x}$ be minimal with this property, and let $\bar{Y}=\tau^{-1}[\zeta]$.\\ 
Let $\alpha$ be the smallest admissible limit of admissible ordinals greater than $\zeta$. Then $\alpha+1<\lambda^{x}$ by Theorem \ref{lambdalimit}.
 
We claim that $x\in L_{\alpha}$. To see this, let $g_{1},g_{2}$ be mutually\\ (Cohen-)generic over $L_{\Sigma^{x}}$ and elements of $\bar{Y}$, which exist by
Lemma \ref{genericlambda}: First, as $\bar{Y}$ is not meager and $\Sigma^{x}$ is countable, $\bar{Y}$ contains a real $g_{1}$ generic over $L_{\Sigma^{x}}$. Again by Lemma \ref{genericlambda},
$\bar{Y}$ contains a real $g_{2}$ generic over $L_{\Sigma^{x}}[g_{1}]$. By standard facts on Cohen-forcing (see e.g. Lemma $30$ of \cite{CS}), $g_1$ and $g_2$ are mutually generic over $L_{\Sigma^{x}}$.

 So $P^{x\oplus g_{1}}\downarrow=1$ and $P^{x\oplus g_{2}}\downarrow=1$. As therefore $V_{\alpha}\models\exists{z}P^{z\oplus g_{1}}\downarrow=1\wedge P^{z\oplus g_{2}}\downarrow=1$, we have
$L_{\alpha}[g_{1}]\models\exists{z}P^{z\oplus g_{1}}\downarrow=1$ and $L_{\alpha}[g_{2}]\models\exists{z}P^{z\oplus g_{2}}\downarrow=1$ by Lemma \ref{relJK}. 

As $g_{1},g_{2}\in Y$ (so $P$ recognizes $x$ relative to $g_1$ and $g_2$) and by absoluteness of computations, the elements $z_{1}\in L_{\alpha}[g_{1}]$ and $z_{2}\in L_{\alpha}[g_{2}]$ witnessing 
$\exists{z}P^{z\oplus g_{1}}\downarrow=1$ and $L_{\alpha}[g_{2}]\models\exists{z}P^{z\oplus g_{2}}\downarrow=1$ must both be equal to $x$, so 
we have $x\in L_{\alpha}[g_{1}]$ and $x\in L_{\alpha}[g_{2}]$, so that finally $x\in L_{\alpha}[g_{1}]\cap L_{\alpha}[g_{2}]=L_{\alpha}$ by mutual genericity of $g_{1},g_{2}$ over $L_{\lambda^{x}}$
and hence over $L_{\alpha}\subseteq L_{\lambda^{x}}$.

$P^{z\oplus g_{1}}$ may not stop in less than $\alpha$ many steps for each $z\in L_{\alpha}$; however, by absoluteness of computations and since $P$ recognizes $z$ from $g_{1}$, it only does so with output $1$ if $z=x$. 
So we have that $L_{\alpha}[g_{1}]\models\forall{z}P^{z\oplus g_{1}}\downarrow=1\leftrightarrow x=z$. 
By the forcing theorem for admissible sets (see e.g. Lemma 32 of \cite{CS}), 
there is a finite $p\subseteq\omega$ such that $p\Vdash\forall{z}P^{z\oplus g_{1}}\downarrow=1\leftrightarrow x=z$ over $L_{\alpha}$.
Consequently, the same holds for every real $g\supseteq p$ which is Cohen-generic over $L_{\alpha+1}$.\\

We can now recognize $x$ by the following procedure: Given a real $z$ in the oracle, we let all ITTM-programs run simultaneously in the oracle $z$ and check the output whenever a computation stops
 until we find a pair $(L_{\xi}[z],g)$\footnote{This is a slight abuse of notation. The first element should
of course actually be a real number coding $L_{\xi}[z]$.} such that $g\supseteq p$, $g$ is generic over $L_{\xi+1}[z]$ and
$\xi$ is a $z$-admissible limit of $z$-admissible ordinals greater than the halting time of $P^{z\oplus g}$. 
(We shall see below that such a pair exists for every $z$ and is an element of $L_{\lambda^{z}}[z]$ and thus
computable by some ITTM in the oracle $z$ so that the search always terminates.) 
Now check (1) whether $z\in L_{\xi}$ and 
(2) whether $L_{\xi}\models\forall{y\subseteq\omega}P^{y\oplus g}\downarrow=1\leftrightarrow y=z$, i.e. whether $z$ is the unique element $y$ of $L_{\xi}$
such that $P^{y\oplus g}\downarrow=1$ in less than $\xi$ many steps. This can be done by evaluating a recursive truth predicate in $L_{\xi}$. We claim that, if either fails, then $z\neq x$, otherwise $z=x$.\\

To see that this procedure works, we first observe that such a pair $(L_{\xi}[z],g)$ always exists and is ITTM-computable from $z$:\\
Given the oracle $z$, pick some $\hat{g}\supseteq p$ Cohen-generic over $L_{\Sigma^{z}+1}[z]$. Then $P^{z\oplus\hat{g}}\downarrow$ in less than $\lambda^{z}$ many steps by Lemma \ref{genericlambda}.
Let $\hat{\zeta}$ be the halting time of $P^{z\oplus\hat{g}}\downarrow$,
 and let $\hat{\alpha}$ be the smallest $z$-admissible limit of $z$-admissible ordinals greater than $\hat{\zeta}$. By Theorem \ref{lambdalimit}, we have $\hat{\alpha}<\lambda^{z}$.
Then $L_{\hat{\alpha}}[z][\hat{g}]\models P^{z\oplus\hat{g}}\downarrow$.\footnote{We may use the generic extension and the relativized constructibility equivalently by Theorem \ref{relforc}.} 
 Again by the forcing theorem over admissible sets, there is $q\subseteq\hat{g}$ such that $q\Vdash P^{z\oplus \hat{g}}\downarrow$ over $L_{\hat{\alpha}}[z]$.
 As $\hat{g}\supseteq p$, $q$ and $p$ are compatible; let $s=q\cup p$. 
Now, if $\hat{g}^{\prime}\supseteq s$ is another generic over $L_{\hat{\alpha}+1}[z]$,
then we still have $L_{\hat{\alpha}}[z][\hat{g}]\models P^{z\oplus\hat{g}^{\prime}}\downarrow$; so the halting time of $P^{z\oplus\hat{g}^{\prime}}$ is less than $\hat{\alpha}$. 
As the projectum in $L[z]$ will drop 
to $\omega$ between $\hat{\alpha}$ and $\lambda^{z}$ (it drops at every halting time, of which $\lambda^{z}$ is the supremum), making $L_{\hat{\alpha}}[z]$ countable in $L_{\lambda^{z}}[z]$, so that the Rasiowa-Sikorski-construction
can be carried out inside $L_{\lambda^{z}}[z]$ with the result that 
such a real $\hat{g}$ will (along with a real coding $L_{\hat{\alpha}}[z]$) be contained in $L_{\lambda^{z}}[z]$.

Now, if $z=x$, then, as $g$ is generic and extends $p$ which forces $P^{x\oplus g}$ to converge to $1$, the procedure will clearly halt with output $1$.
 So assume towards a contradiction that our procedure stops with output $1$ in the oracle $z\neq x$. 
Let $(L_{\xi}[z],g)$ be the pair found in the execution of the procedure. In particular, this means that $z\in L_{\xi}$. We distinguish two cases:\\

\textbf{Case 1}: $\xi\geq\alpha$. Then $g\supseteq p$, being generic over $L_{\xi+1}$, is also generic over $L_{\alpha+1}$. 
Hence $P^{x\oplus g}\downarrow=1$ in less than $\alpha\leq\xi$ many steps by the choice of $p$ and the fact that $x\in L_{\alpha}\subseteq L_{\xi}$.
So $z$ is not the only element $y$ of $L_{\xi}$ with $P^{y\oplus g}\downarrow=1$ in less than $\xi$ many steps, contradicting the assumption that our procedure stopped with output $1$.\\

\textbf{Case 2}: $\xi<\alpha$. As $z\in L_{\xi}$ and $\xi$ is admissible, we have $L_{\xi}[z]=L_{\xi}$
By the forcing theorem over admissible sets once more, there is a finite $q\subseteq g$ such that $q\Vdash P^{z\oplus g}\downarrow=1$ over $L_{\xi}$ (thus in less than $\xi$ many steps).
As $g\supseteq p$, $q$ and $p$ are compatible; let $s=q\cup p$. 
Now pick $\bar{g}\supseteq s$ generic over $L_{\alpha+1}$.
 Then, as $\bar{g}\supseteq p$, $x$ should be the only element $y$ of $L_{\alpha}$ with $P^{y\oplus g}\downarrow=1$ in less than $\alpha$ many steps; 
but as $\bar{g}\supseteq q$, we also have that $P^{z\oplus\bar{g}}\downarrow=1$ in less than $\xi<\alpha$ many steps, contradicting the assumption that $z\neq x$.\\

Hence the procedure identifies $x$, as desired.
\end{proof}

As in the ITRM-part, we can deduce:

\begin{corollary}{\label{ITTMnonmeager}}
\begin{enumerate}
\item Let $Y\subseteq[0,1]$ be non-meager, and let $x\subseteq\omega$ be uniformly ITTM-recognizable in $Y$. Then $x$ is ITTM-recognizable.
\item Assume that every $\bf\Sigma^{1}_{2}$-set of reals has the Baire property. Let $Y$ be a non-meager set, $x\subseteq\omega$ and assume that, for every $y\in Y$, there is some ITTM-program $P$
such that $P^{z\oplus y}\downarrow=1$ if and only if $z=x$. Then $x$ is ITTM-recognizable.
\end{enumerate}
\end{corollary}

\subsection{Other Machines}

Other notable machine models of infinitary computability include 
$\alpha$-Turing machines (see \cite{KS}), $\alpha$-register machines (see \cite{Ko2} and \cite{Ca}), and ordinal Turing machines (OTMs) (see \cite{Ko1}) with and without ordinal parameters. 

Concerning OTMs without parameters, we have that, by \cite{Ca}, recognizability
equals computability, and the proof relativizes: Roughly, there is a non-halting OTM-program $Q$ that, given the oracle $x$, enumerates $L[x]$. By Shoenfield absoluteness, if $P$ is a program
that parameter-freely recognizes $y$ relative to $x$, then, as `There is a real $z$ such that $P^{z\oplus x}\downarrow=1$' is a $\Sigma_{1}$-statement which, as it holds in $V$ by assumption, must also hold in $L[x]$. One can thus
compute $y$ in the oracle $x$ by a parameter-free OTM by enumerating $L[x]$, running $P^{z\oplus x}$ whenever a new real $z$ is produced and halting and outputting $z$ once $P^{z\oplus x}\downarrow=1$.
As the claim that parameter-free OTM-computability from all elements of a non-meager or a positive set of oracles implies OTM-computability is independent of ZFC by section 2.1 of \cite{CS}, the same holds for the recognizable analogue.

Recognizability for OTMs with ordinal parameters is a more delicate issue.
It is shown in \cite{Ca} that $0^{\sharp}$ is OTM-recognizable in the parameter $\omega_{1}$, and the same argument can be applied to show the recognizability of reals even more remote from $L$. One of the results
of \cite{CS1} is that, under the assumption that $M_{1}^{\sharp}$ exists, the closure of $\emptyset$ under relativized parameter-OTM recognizability (for real numbers) coincides with $\mathbb{R}^{M_{1}}$,
where $M_{1}$ is the mouse for a Woodin cardinal. Trivially, as every constructible real is computable and hence recognizable by some parameter-OTM, the claim that for parameter-OTMs, recognizability
from all elements of a non-meager set implies recognizability holds in $L$. 
On the other hand, we have: 

\begin{lemma}{\label{whomgennonrec}}
 Let $\mathbb{P}$ be a weakly homogenous notion of forcing (i.e. for
any two conditions $p,q$, there is an automorphism $\pi$ of $\mathbb{P}$ such that $\pi(p)$ and $q$ are compatible ), $M\models\text{ZFC}$ a transitive model containing $\mathbb{P}$, $G$ a $\mathbb{P}$-generic filter over $M$ and $x\subseteq\omega$
such that $x\in M[G]\setminus M$. Then $x$ is not recognizable.
\end{lemma}
\begin{proof}
 Assume otherwise, and fix a weakly homogenous forcing notion $\mathbb{P}$, a transitive $M\models\text{ZFC}$ containing $\mathbb{P}$, a $\mathbb{P}$-generic filter $G$ over $M$ and
a real $x\in M[G]\setminus M$ such that, for some OTM-program $P$ and some ordinal $\alpha$, $P$ recognizes $x$ in the parameter $\alpha$. By absoluteness of computations, this means 
in particular that $P$ recognizes $x$ in the parameter $\alpha$ in $M[G]$. 

By the forcing theorem, there is then a condition $p\in G$ such that
$p$ forces that $\check{P}$ recognizes $\dot{x}$ in the parameter $\check{\alpha}$,
where we denote by $\check{z}$ the canonical name for $z$.
If $p$ would decide every bit of $x$, then we would have $x\in M$, contradicting our assumption. Let $q_{0},q_{1}$ be two strengthenings of $p$ that decide some bit differently, say $q_{0}\Vdash\dot{x}(i)=0$ and
$q_{1}\Vdash\dot{x}(i)=1$ with $i\in\omega$, and let $\pi$ be an automorphism of $\mathbb{P}$ such that $\pi(q_{1})$ and $q_{0}$ are compatible. 
Let $G^{\prime}$ be a filter containing $q_{0}$ and $\pi(q_{1})$. Then, as $q_{0}$ and $q_{1}$ strengthen $p$ which forces that $\dot{x}$ is recognized by $P$ in the parameter $\alpha$, we have 
$q_{0}\Vdash\check{P}^{\dot{x}}(\check{\alpha})\downarrow=1$, $\pi(q_{1})\Vdash \check{P}^{\pi(\dot{x})}(\check{\alpha})\downarrow=1$ and
$\pi(q_{1})\Vdash\pi(\dot{x})(\pi(i))=\pi(1)$, i.e. $\pi(q_{1})\Vdash\pi(\dot{x})(i)=1$. Hence in $M[G^{\prime}]$, we have $x_{0}:=\dot{x}^{G^{\prime}}\neq\pi(\dot{x})^{G^{\prime}}=:x_{1}$,
but also both $P^{x_{0}}(\alpha)\downarrow=1$ and $P^{x_{1}}(\alpha)\downarrow=1$. On the other hand, as $G^{\prime}$ contains $p$, $P$ recognizes some real number in the parameter $\alpha$, a contradiction.
This shows that no real that is added by a weakly homogenous forcing is recognizable. 
\end{proof}

\begin{defini}{\label{laverdef}}
 Let $\mathbb{L}$ denote Laver forcing (see e.g. Definition 28.15 of \cite{Je}). Thus, $\mathbb{L}$ consists of trees $p$ of finite sequences of natural numbers with the properties that
(i) there is a maximal element $t_{p}\in p$, called the `stem' of $p$, such that every $s\in p$ extends $t_{p}$ or is an initial segment thereof and (ii) every $t\in p$ that extends $t_{p}$ has infinitely many successors that are by
exactly one element longer than $t$. The partial ordering on $\mathbb{L}$ is just the subset relation.
\end{defini}

\begin{lemma}{\label{laverwh}}
 Laver forcing is weakly homogenous.
\end{lemma}
\begin{proof}
Suppose that $p$ and $q$ are conditions of Laver-forcing. We want to find Laver conditions $p^{\prime}\leq p$, $q^{\prime}\leq q$ and an automorphism
$\pi$ of $\mathbb{P}$ such that $\pi(p^{\prime})=q^{\prime}$. If the stems of $p$ and $q$ have different lengths, we can cut off branches from the condition with the shorter stem until the lengths are equal, which will strengthen this condition.
We may thus assume without loss of generality that the stems of $p$ and $q$ are of equal length.

We now thin out $p,q$ to $p^{\prime},q^{\prime}\in\mathbb{L}$ without changing the length of the stem such that for each $i\in\omega$, the $i$th levels of $p^{\prime}$ and $q^{\prime}$ have no common label and each label appears at most once.
It is easy to see that, when $\beta:\omega\rightarrow\omega$ is bijective, then $\pi^{i}_{\beta}:\mathbb{L}\rightarrow\mathbb{L}$ that applies $\beta$ to each label in the $i$th level of a tree,
is an automorphism of $\mathbb{L}$. But now, by chosing appropriate $\beta_{i}$ for each $i\in\omega$ and applying $\pi:=\bigcup_{j\in\omega}\pi^{j}_{\beta_{j}}$, we have an automorphism of $\mathbb{L}$
that maps $p^{\prime}$ to $q^{\prime}$. Thus, we have $\pi(p)\supset\pi(p^{\prime})=q^{\prime}\subseteq q$, hence $q^{\prime}$ is a common strengthening of $q$ and $\pi(p)$, so that $\pi(p)$ and $q$ are compatible.


\end{proof}

\begin{thm}{\label{laver}}
There is a generic extension $L[G]$ of $L$ such that the generic reals form a comeager sets, yet none of them is parameter-OTM-recognizable.
\end{thm}
\begin{proof}


By \cite{G}, Laver forcing is minimal; hence, when $x,y$ are generic, they are constructible in each other, i.e. $x\in L[y]$ and $y\in L[x]$. As the parameter-OTM-computable reals in the oracle $y$
are exactly the elements of $L[y]$ (see e.g. Lemma 17 of \cite{CS}), this means that all generics are parameter-OTM-computable, and hence in particular recognizable, from each other. 

On the other hand, Laver forcing is weakly homogenous by Lemma \ref{laverwh}. 

Now let $G$ be generic for Laver forcing over $L$ and consider $L[G]$. By Lemma \ref{whomgennonrec}, it follows that no real that is added through the forcing is recognizable. By
Theorem 7.3.28 of \cite{BJ}, Laver forcing makes the set of ground model reals meager, so that the added elements form a comeager set. Any real in $L[G]\setminus L$ is hence recognizable relative to all other such reals,
but not itself recognizable.

Therefore, the claim that parameter-OTM-recognizability from all elements of a comeager set of oracles implies parameter-OTM-recognizability fails in $L[G]$.
\end{proof}

By Theorem \ref{laver} and the remark preceeding it, we get:

\begin{corollary}{\label{relrecogindep}}
It is independent of ZFC whether parameter-OTM-recognizability from all elements of a comeager set of oracles implies parameter-OTM-recognizability.
\end{corollary}


Recognizability for $\alpha$-register machines was considered briefly in \cite{Ca}, where it turned out that the existence of lost melodies depends on $\alpha$. We do not know for which $\alpha$ the analogue of our statement holds.

\section{The Recognizable Jump Operator}


A notion of computability is commonly accompanied by a corresponding jump operator; a jump operator can roughly be seen as the set of programs that compute something in the sense of the notion of computability in question.
This motivates the introduction of a jump operator for recognizability. This jump operator will turn out to be strongly connected to $\Sigma_{1}$-stability and is conceptually stable in the sense that the recognizable jumps
for ITRMs and ITTMs, which are otherwise very different in strength, are primitive recursively equivalent.

In this section, we will, besides ITRMs and ITTMs, also consider Ordinal Turing Machines (OTMs), introduced in \cite{Ko1}. Unless stated otherwise, we will consider OTMs without ordinal parameters.

It is easy to see that the two variants of the jump operator given by (1) $x^{\prime}:=\{i\in\omega:\forall{j\in\omega}P^{x}_{i}(j)\downarrow\}$ and (2) $x^{\prime}:=\{i\in\omega:P^{x}(0)\downarrow\}$ are equivalent for
the models of computability discussed here (i.e. ITRMs, ITTMs, OTMs). Namely, to reduce (1) to (2), consider, given the index $i$, the program $Q$ that, for any input in the first register,
 lets $P_{i}^{x}(j)$ run successively for all $j\in\omega$. Then $Q^{x}(0)$ halts if and only if $P_{i}^{x}(j)$ halts for every $j\in\omega$, and an index for $Q$ is easily Turing-computable from $i$.
 To reduce (2) to (1), given index $i$, consider the program $Q$ that, for any input in the first register, runs $P_{i}^{x}(0)$. Then $P_{i}^{x}(0)$ halts if and only if $Q^{x}(j)$ halts for every $j\in\omega$.
We can thus say that the jump operator for a model of infinitary computability sends a real $x$ to the set of all indices $i\in\omega$ such that $P_{i}^{x}$ computes a real number.

Analogously, we now define the `recognizable jump operator', or $r$-jump, $x^{r}$ of a real to be the set of all indices $i\in\omega$ such that $P_{i}^{x}$ recognizes a real:

\begin{defini}{\label{recjump}}
Let $M$ be ITRM, ITTM or OTM, $x\subseteq\omega$. Fix a natural enumeration $(P_{i,M}:i\in\omega)$ of the $M$-programs. 
Then $x^{r}_{M}$, the $r$-jump of $x$ (for $M$), is defined as $\{i\in\omega:\exists{y\subseteq\omega}\forall{z\subseteq\omega}P_{i,M}^{z\oplus x}\downarrow=\delta(z,y)\}$.
We can iterate this operator by setting $r-x^{0}_{M}:=0$ and $r-x^{i+1}_{M}:=(r-x^{i}_{M})^{r}_{M}$. Transfinite iterations are also possible as for the Turing jump, but will not be considered here.
When $M$ is clear from the context, we drop it.
\end{defini}

Many of the following theorems hold for ITRMs, ITTMs and OTMs. To avoid repitions, we use an index $M$ to denote, unless stated otherwise, ITRMs, ITTMs and OTMs for the rest of the section.
Thus, $\equiv_{M}$ means computational equivalence in the sense of the model $M$.

We start by noting that the recognizable jump enjoys the appropriate amount of stability to be expected of a jump operator:

\begin{prop}
Let $x,y\subseteq\omega$ such that $x\equiv_{M}y$. Then $x^{r}_{M}\equiv_{M}y^{r}_{M}$.
\end{prop}
\begin{proof}
 We show that $x^{r}_{M}\leq_{M}y^{r}_{M}$, the other direction following by symmetry. So assume that $y^{r}_{M}$ is given in the oracle and we want to determine whether $(P_{i}^{M})^{x}$ recognizes a real number.
Let $Q$ be an $M$-program that computes $x^{r}_{M}$ from $y^{r}_{M}$. From $P_{i}^{M}$ and $Q$, it is easy to obtain (primitive recursively, in fact) an $M$-program $Q^{\prime}_{i}$ that, given the oracle $z_{0}\oplus z_{1}$, works by first applying
$Q$ to $z_{0}$ and then, after $Q^{z_{0}}$ halts (if it does) with output $z$, running $(P_{i}^{M})^{z\oplus z_{1}}$. Now, $Q_{i}^{\prime}$ recognizes a real relative to $y$ if and only if $P_{i}^{M}$ recognizes a real relative to $x$;
but whether $Q_{i}^{\prime}$ recognizes a real number can simply be determined by using $y^{r}_{M}$.
\end{proof}

\begin{prop}{\label{recjumpabs}}
$0^{r}_{M}$ is absolute between $V$ and $L$ for $M\in\{\text{ITRM},\text{ITTM}\}$.
\end{prop}
\begin{proof}
 That a program $P$ does not recognize a real number means that one of the following holds: (1) There is a real number $x$ such that $P^{x}\uparrow$ (2) There is a real number $x$ such that $P^{x}\downarrow\notin\{0,1\}$
(3) There is no real number $x$ such that $P^{x}\downarrow=1$ (4) There are different real numbers $x,y$ such that $P^{x}\downarrow=1$ and $P^{y}\downarrow=1$.
So (2) and (4) are set-theoretical $\Sigma_1$-statements and hence absolute between $V$ and $L$. Concerning (3), if $L$ contains a real number $x$ such that $P^{x}\downarrow=1$, then, by absoluteness
of computations, so does $V$. On the other hand, if $V\models\exists{x}P^{x}\downarrow=1$, then, by Shoenfield absoluteness, $L\models\exists{x}P^{x}\downarrow=1$ and by absoluteness of computations, $L$ contains
some $y$ such that $P^{y}\downarrow=1$. Finally, (1) is $\Sigma_{1}$-expressible for ITRMs as `There is $y\subseteq\omega$ such that $L_{\omega_{\omega}^{\text{CK},y}}[y]\models P^{y}\uparrow$' and
for ITTMs as `There is $y\subseteq\omega$ such that $L_{\lambda^{y}}[y]\models P^{y}\uparrow$' together with Theorem \ref{ITTMchar}.
\end{proof}

We observe that the computable jump of $x$ reduces to the recognizable jump of $x$, so that the latter is not computable from $x$:

\begin{prop}{\label{cjumpfromrjump}}
$x^{\prime}_{M}\leq_{T}x^{r}_{M}$ (where $\leq_{T}$ denotes Turing reducibility).
\end{prop}
\begin{proof}
Let $i\in\omega$. To test, using $x^{r}_{M}$, whether $P_{i}^{x}(0)$ halts, we compute from $i$ a code for the program $Q$ that does the following:
First, $Q$ runs $P^{x}(0)$. Once $P^{x}(0)$ has stopped (if ever), $Q$ checks whether $x=0$ and returns $1$ if $x=0$ and otherwise $0$.
Clearly, $Q$ recognizes a real (namely $0$) if and only if $P^{x}(0)$ halts. And an index for $Q$ is easily Turing-computable from $i$.
\end{proof}

A crucial property of the computable jump is that $x^{\prime}_{M}$ is not $M$-computable from $x$. The next goal is to show that the same holds for the $r$-jump.

\begin{defini}{\label{stable}}
(See \cite{Ba}) An ordinal $\alpha$ is $1$-stable if and only if $L_{\alpha}\prec_{\Sigma_{1}}L$. $\alpha$ is $\Sigma_1$-fixed if and only if there is some $\Sigma_1$-statement $\phi$
such that $\alpha$ is minimal with the property that $L_{\alpha}\models\phi$. For $\iota\in On$, $\sigma_{\iota}$ denotes the $\iota$th $1$-stable ordinal.
The first $1$-stable ordinal, $\sigma_{0}$, is also denoted $\sigma$. This notation relativizes to real parameters in the obvious way.
\end{defini}

\begin{lemma}{\label{barwise}}\ \\
(1) $\sigma$ is the supremum of the $\Sigma_{1}$-fixed ordinals.\\
(2) If $\alpha$ is $1$-stable, then $\alpha$ is recursively inaccessible, i.e. an admissible limit of admissible ordinals.\\
(3) $L_{\sigma}$ is the set of all $x$ that are parameter-free $\Sigma_{1}$-definable in $L$.
\end{lemma}
\begin{proof}
 See Corollary $V.7.9$ and Corollary $V.7.6$ of \cite{Ba}. 
\end{proof}

\begin{lemma}{\label{recogsup}}
 $\sigma=\text{sup}\{\alpha\in\text{On}:\exists{x\in\text{RECOG}_{M}}(x\notin L_{\alpha}\wedge x\in L_{\alpha+1})\}$. 
In particular, we have $\text{RECOG}_{M}\subseteq L_{\sigma}$.
\end{lemma}
\begin{proof}
 This is done in Theorem $27$ of \cite{Ca1} for ITRMs, but the same argument works for ITTMs and parameter-free OTMs: If $x\subseteq\omega$ is $M$-recognizable, then, by Shoenfield absoluteness, it is constructible. 
Now, if $P$ recognizes $x$, then $\exists{y}P^{y}\downarrow=1$ is a $\Sigma_1$-definition must become true for the first time in some $L_{\alpha}$ with $\alpha<\sigma$, so that $x\in L_{\sigma}$. On the other hand, 
if $\alpha$ is minimal such that $L_{\alpha}\models\phi$ for some $\Sigma_{1}$-statement $\phi$, then $L_{\alpha+1}$ will contain a $<_{L}$-minimal real coding $L_{\alpha}$ which can be recognized as the $<_{L}$-minimal code
of an $L$-level in which $\phi$ holds.
\end{proof}

As one would expect, the recognizable jump of a real number $x$ transcends recognizability relative to $x$:

\begin{thm}{\label{rjumpnonrec}}
Let $x\subseteq\omega$. Then $x^{r}_{M}$ is not $M$-recognizable relative to $x$.
\end{thm}
\begin{proof}
 We prove this for $x=0$. The proof relativizes to arbitrary oracles. 

By Lemma \ref{recogsup}, it suffices to show that $0^{r}_{M}\notin L_{\sigma}$. So assume otherwise for a contradiction.
By Lemma \ref{barwise} then, let $\phi$ be a $\Sigma_{1}$-formula such that $0^{r}_{M}$ is unique with the property that
$L\models\phi(0^{r}_{M})$. Clearly, $0^{r}_{M}$ is an infinite set. Hence the function $f:\omega\rightarrow\omega$
sending $i$ to the $i$th element of $0^{r}_{M}$ is total and definable in the parameter $0^{r}_{M}$ and hence also contained in $L_{\sigma}$.

Now, let $g:\omega\rightarrow\sigma$ be the function that sends $i\in\omega$ to the smallest $\alpha\in\text{On}$ such that
$L_{\alpha}\models\exists{x}P_{f(i)}^{x}\downarrow=1$. As $P_{f(i)}^{x}\downarrow=1$ is $\Sigma_1$ in the parameter $0^{r}_{M}$, 
such an $\alpha$ is clearly $\Sigma_1$-fixed and hence below $\sigma$ by Lemma \ref{barwise}. Moreover, $L_{\alpha}$ will contain
the unique real $x$ such that $P_{f(i)}^{x}\downarrow=1$. Hence, the supremum of these $\alpha$ will be $\sigma$ by Lemma \ref{recogsup}.
We show that $g$ is $\Sigma_{1}$-definable over $L_{\sigma}$. This will be the desired contradiction, as $g$ will then be a $\Sigma_1$-definable
total function mapping $\omega<\sigma$ cofinally into $\sigma$, contradicting the fact that $\sigma$ is admissible by Lemma \ref{barwise}.
But $g(i)=\alpha$ can be written as 
$\exists{x,y,j}[(\phi(x)\wedge j\in x\wedge|x\cap j|=i)\wedge(y=L_{\alpha}\wedge y\models(\exists{z}P_{j}^{z}\downarrow=1)\wedge y\models(\forall{\gamma\in\alpha}(L_{\gamma}\not\models(\exists{z}P_{j}^{z}\downarrow=1))))]$ 
(where the first conjunct says that $j$ is the $i$th element of $0^{r}_{M}$ while the second expresses that $\alpha$ is minimal with the property that $L_{\alpha}$ believes in the existence of some real $z$ with $P_{f(i)}^{z}\downarrow=1$),
which is $\Sigma_1$. 
\end{proof}

We can also show the unrecognizability of the recognizable jump more directly by a diagonalization argument that works rather generally for models of computation that allow universal programs (which includes ITTMs, OTMs, OTMs with a fixed parameter $\alpha$,
etc. but not ITRMs):

\begin{thm}{\label{recjumpnonrec}}
 The recognizable jump $0^{r}$ is not recognizable.
\end{thm}
\begin{proof}
 Assume otherwise, so that $0^{r}\in\text{RECOG}$. Let $(P_{i}:i\in\omega)$ enumerate the programs. Denote by $j_{i}$ the $i$th element of $0^{r}$ for $i\in\omega$ (so $j_{i}$ is the index of the $i$th recognizing program) and
by $x_{i}$ the real recognized by $P_{j_{i}}$. We note that $x:=\oplus_{i\in\omega}x_{i}$ is recognizable relative to $0^{r}$ by observing that the following procedure recognizes $x$ relative to $0^{r}$:
Given $y=\oplus_{i\in\omega}y_{i}$ in the oracle, we perform the following for every $i\in\omega$: First, we find $j_{i}$ using $0^{r}$. Then, we run $P_{j_{i}}^{y_{i}}$. As $j_{i}\in 0^{r}$, $P_{j_{i}}^{y_{i}}$ must stop with output $0$ or $1$.
If the output is $0$, then $y\neq x$ and we stop with output $0$; otherwise, we continue. When we have run through all $i\in\omega$ in this way, then $x=y$.

It follows that $z:=x\oplus 0^{r}$ is recognizable (the second component $0^{r}$ is recognizable by assumption, the first then relative to the second by the above). We will now construct a nonrecognizable real $\bar{z}$ from $z$ by diagonalizing against
$(x_{i}:i\in\omega)$; as $\bar{z}$ will be seen to be recognizable if $z$ is, this will be a contradiction.

Let $p:\omega\times\omega\rightarrow\omega$ denote Cantor's pairing function.
The $0$th bit of $x_{i}$ is represented by the $2p(i,0)$th bit of $z$. We now define $\bar{z}$ by letting $\bar{z}(2p(i,0)):=1-z(2p(i,0))$ if $x_{i}(0)=x_{i}(2p(i,0))$ and $\bar{z}(j)=z(j)$ otherwise.
Note that the so constructed $\bar{z}$ will hence differ from $x_{i}$ in the $2p(i,0)$th bit for all $i\in\omega$: If $x_{i}(0)\neq x_{i}(2p(i,0))$, then $\bar{z}(2(p(i,0))=x_{i}(0)\neq x_{i}(2p(i,0))=z_{i}(2p(i,0))=\bar{z}(2p(i,0))$,
and if $x_{i}(0)=x_{i}(2p(i,0))$, then $\bar{z}(2(p(i,0))=x_{i}(0)=x_{i}(2p(i,0))=z_{i}(2p(i,0))\neq 1-z_{i}(2p(i,0))=\bar{z}(2p(i,0))$. As each recognizable real is among the $x_{i}$ and $\bar{z}$ is different from all the $x_{i}$,
$\bar{z}$ cannot be recognizable.

On the other hand, given that $0^{r}$ is recognizable, the following procedure recognizes $\bar{z}$: Given $y=y_{1}\oplus y_{2}$ in the oracle, first check whether $y_{2}=0^{r}$. If not, then $y\neq\bar{z}$.
Otherwise, let $y_{1}=\oplus_{i\in\omega}y_{1,i}$ and run the following procedure for each $i\in\omega$: First, check whether $y_{1,i}(0)=y_{1,i}(2p(i,0))$. If yes, then $y\neq\bar{z}$. Otherwise,
let $\tilde{y}_{1,i}(0)=1-y_{1,i}(0)$ and $\tilde{y}_{1,i}(j)=y_{1,i}(j)$ for $j>0$ and check whether $P_{j_{i}}^{y_{1,i}}\downarrow=1$ or $P_{j_{i}}^{\tilde{y}_{1,i}}\downarrow=1$. If not, then $y\neq\bar{z}$.
If, on the other hand, we have run through all natural numbers in this manner, then $y=\bar{z}$.

So it follows that $\bar{z}$ is both recognizable and unrecognizable, a contradiction. Thus $0^{r}$ is not recognizable.
\end{proof}

Iterating the argument for Theorem \ref{rjumpnonrec}, we get:

\begin{corollary}{\label{stableandrecogjump}}
 For $i\in\omega$ and $M\in\{\text{ITRM}, \text{ITTM}\}$, we have $r-0^{i}_{M}\in L_{\sigma_{i+1}}\setminus L_{\sigma_{i}}$. In particular, $r-0^{i}_{M}$ is not M-recognizable from $r-0^{i-1}_{M}$ for $i>0$. 
\end{corollary}
\begin{proof}
We adapt the proof of Theorem \ref{rjumpnonrec}. As there, we can see that $r-0^{k}_{M}\notin L_{\sigma_{k}}$, using that, by Corollary 7.9 of \cite{Ba}, $L_{\sigma_{j+1}}$ consists of those elements of $L$
that are definable by ordinal parameters $\leq\sigma_{j}$ for all $j\in\omega$.
If we had $r-0^{k}_{M}\in L_{\sigma_{k}}$, the function $g$ sending each $i\in\omega$ to the smallest ordinal $\alpha_{i}$ such that $L_{\alpha_{i}}$ contains some computation
of $P^{x\oplus r-0^{k}_{M}}_{f(i)}$ that converges to $1$, where $x\subseteq\omega$ and $f(i)$ denotes the $i$th element of $r-0^{k}_{M}$ would be $\Sigma_{1}$-definable over $L_{\sigma_{k}}$ and cofinal in $\sigma_{k}$, so that
$\sigma_{k}$ could not be admissible, contradicting Corollary 7.6 of \cite{Ba}.

To see that $r-0^{i}_{M}\in L_{\sigma_{i+1}}$, we proceed inductively, using the assumption $r-0^{i-1}_{M}\in L_{\sigma_{i}}$ (for $i>0$) to show that
$r-0^{i}_{M}$ is definable over $L_{\sigma_{i}}$ and hence contained in $L_{\sigma_{i+1}}$.\\

\textbf{Claim}: For every $j\in\omega$, the program $P_{j}$ recognizes a real number relative to $r-0^{i-1}_{M}$ if and only if $L_{\sigma_i}$ believes that it does.

\begin{proof} For ITRMs and ITTMs, the property that $P^{x}\uparrow$ is $\Sigma_1$-expressable in the parameter $x$
by stating that $P^{x}$ does not halt in $\omega_{\omega}^{\text{CK},x}$ many steps or (by Theorem \ref{ITTMchar}) that there is a minimal triple $(\alpha,\beta,\gamma)$ 
with $L_{\alpha}[x]\prec_{\Sigma_{1}}L_{\beta}[x]\prec_{\Sigma_{2}}L_{\gamma}[x]$ and $P^{x}$ does not halt in
$\alpha$ many steps, respectively. Hence $\exists{x}P^{x\oplus r-0^{i-1}_{M}}\uparrow$ is $\Sigma_{1}$ in the parameter $r-0^{i-1}_{M}$ and thus absolute between $L_{\sigma_{i}}$ and the real world for every $P$.
Thus, if there was some real number $x$ for which $P^{x\oplus r-0^{i-1}_{M}}$ did not halt, such an $x$ would be contained in $L_{\sigma}$. Hence, the statement `For every real $x$, $P^{x\oplus r-0^{i-1}_{M}}$ halts'
is absolute for $L_{\sigma_{i}}$. Similarly, we get the absoluteness of $\phi_{0}(P):=$`For every $x$, $P^{x\oplus r-0^{i-1}_{M}}$ halts with output $0$ or $1$' for $L_{\sigma_{i}}$.
The statement $\phi_{1}(P):=\exists{x}P^{x\oplus r-0^{i-1}_{M}}\downarrow=1$ is $\Sigma_{1}$ in $r-0^{i-1}_{M}$ and thus absolute between $L_{\sigma_{i}}$ and $L$ by the stability of $\sigma_{i}$, and by the recursive inaccessibility of $\sigma_{i}$ and
Theorem \ref{relJK} also between $L$ and $V$. Finally, $\phi_{2}(P):=\exists{x,y}(x\neq y\wedge P^{x\oplus r-0^{i-1}_{M}}\downarrow=1\wedge P^{y\oplus r-0^{i-1}_{M}}\downarrow=1)$ is also $\Sigma_{1}$ in $r-0^{i-1}_{M}$ and hence absolute for $L_{\sigma_{i}}$
for the same reason. Hence the statement `$P$ recognizes some real number from $r-0^{i-1}_{M}$' is equivalent with $\phi_{0}(P)\wedge \phi_{1}(P)\wedge\neg\phi_{2}(P)$ and therefore absolute for $L_{\sigma_{i}}$, as desired.
Thus $k\in r-0^{i}_{M}$ if and only if $L_{\sigma_{i}}\models k\in r-0^{i}_{M}$ for all $k\in\omega$
\end{proof}

By the claim, $r-0^{i}_{M}$ is definable over $L_{\sigma_{i}}$ as the set $\{k\in\omega:\phi_{0}(P_{k})\wedge\phi_{1}(P_{k})\wedge\neg\phi_{2}(P_{k})\}$. Hence $r-0^{i}_{M}\in L_{\sigma_{i+1}}$.



\end{proof}

\textbf{Remark}: For parameter-free OTMs, we have no $\Sigma_{1}$-definable bound on the halting times so that $P^{x}\uparrow$ might fail to be $\Sigma_1$-expressible; hence the argument doesn't work in this case.
Clearly, $0^{r}_{\text{OTM}}$ is constructible when $V=L$, but we currently do not even know whether the constructibility of $0^{r}_{\text{OTM}}$ follows from ZFC alone.

\begin{prop}{\label{easyreduction}}
 $0^{r}_{ITRM}\leq_{T}0^{r}_{ITTM}\leq_{T}0^{r}_{OTM}$.
\end{prop}
\begin{proof}
It is easy to see that ITRM-programs can be simulated by ITTM-programs and ITTM-programs can be simulated by OTM-programs and there are recursive maps $f,g$ sending each ITRM-program P to an ITTM-program $P^{\prime}$
with the same behaviour for all oracles and each ITTM-program $Q$ to an OTM-program $\hat{Q}$ with the same behaviour for all oracles, respectively.
\end{proof}

The following results concern the strength of the recognizable jump operator.

We have shown above how to retrieve the halting information from $0^{(r)}$ for ITRMs. This can in fact be iterated:

\begin{thm}{\label{iteratedhalting}}
 From $0^{r}_{\text{ITRM}}$ and $i\in\omega$, we can Turing-compute $0^{(i)}_{\text{ITRM}}$ (i.e. the $i$th iterated halting problem) for ITRMs.
\end{thm}
\begin{proof}
 Recall that when $x$ is ITRM-recognizable, then so is $x^{\prime}$ (see Corollary $35$ of \cite{Ca4}) and hence so are all finite iterations of the jump of $x$.
Suppose we want to compute $0^{(i+1)}_{\text{ITRM}}$ from $0^{(r)}$. Let $P$ be an ITRM-program recognizing $0^{(i)}$.
Then, using $P$, it is easy to find (effectively) for every $j\in\omega$ an ITRM-program $Q_{j}$ such that $Q_{j}^{x}\downarrow=1$ if and only if $x$ is the $<_{L}$-minimal code
of a halting computation of $P_{j}^{0^{(i)}_{\text{ITRM}}}$. In fact, an index $f(j)$ for $Q_{j}$ can be obtained primitive recursively from $j$.
Now $Q_{j}$ will recognize a real number if and only if $P_{j}^{0^{(i)}}$ halts, i.e. $f(j)\in 0^{(r)}$ if and only if $j\in 0^{(i+1)}$. 
Hence $0^{(i+1)}_{\text{ITRM}}$ can be obtained by a Turing-computation from $0^{(r)}_{\text{ITRM}}$.
\end{proof}


We note here the somewhat curious fact that there are ITRM-unrecognizable real numbers $x$ between $0$ and $0^{\prime}_{\text{ITRM}}$ 
such that their ITRM-jump $x^{\prime}_{\text{ITRM}}$ is recognizable:

\begin{lemma}{\label{onlyjumprecog}}
 There is an ITRM-unrecognizable real $x$ such that $x^{\prime}$ is recognizable. In fact, $x$ can be taken to be strictly below $0^{\prime}_{\text{ITRM}}$.
\end{lemma}
\begin{proof}
Trivially, we have $0^{\prime}_{\text{ITRM}}\leq_{\text{ITRM}}x^{\prime}_{\text{ITRM}}$ for any $x$.

Let $x$ be Cohen-generic over $L_{\omega_{\omega}^{\text{CK}}}$ and ITRM-computable from $0^{\prime}_{\text{ITRM}}$. It follows from Lemma $29$ of \cite{Ca4} that such an $x$ exists.
From Theorem $16$ of the same paper, it follows that $x<_{\text{ITRM}}0^{\prime}_{\text{ITRM}}$. Clearly, $x\notin L_{\omega_{\omega}^{\text{CK}}}$, so $x$ is not ITRM-computable.

We also have that $\omega_{\omega}^{\text{CK},x}=\omega_{\omega}^{\text{CK}}$. Now if $i\in\omega$ and $P_{i}$ is the $i$th ITRM-program using $n$ registers, 
then $P_{i}^{x}$ stops if and only if there is a forcing condition
$p$ such that $p\Vdash P_{i}^{\check{x}}\downarrow$ over $L_{\omega_{n+1}^{\text{CK}}}$. But from $0^{\prime}_{\text{ITRM}}$, we can compute a code for $L_{\omega_{n+1}^{\text{CK}}}$, which suffices both
to evaluate the statement that a condition $p$ forces a certain statement over $L_{\omega_{n+1}^{\text{CK}}}$, and to exhaustively search for such a condition. If none is found, then $P_{i}^{x}$ will not halt,
otherwise it will. This hence allows us to solve the halting problem relative to $x$ in the oracle $0^{\prime}_{\text{ITRM}}$, so that $x^{\prime}_{\text{ITRM}}\leq_{\text{ITRM}}0^{\prime}_{\text{ITRM}}$. 
Hence $x^{\prime}_{\text{ITRM}}\equiv_{\text{ITRM}}0^{\prime}_{\text{ITRM}}$.

As $0^{\prime}_{\text{ITRM}}$ is ITRM-recognizable by section $4$ of \cite{Ca2}, it follows from Proposition \ref{recogstable} that $x^{\prime}_{\text{ITRM}}$ is also ITRM-recognizable, so $x$ is as desired.
\end{proof}

With a similar idea to that of the proof of Lemma \ref{onlyjumprecog}, we also get the following, stronger statement that the jump of every real ITRM-computable from, but not ITRM-equivalent to 
$0^{\prime}_{\text{ITRM}}$ is `ITRM-low', i.e. has its jump equivalent to $0^{\prime}_{\text{ITRM}}$:

\begin{corollary}{\label{nojumpinv}}
Let $x<_{\text{ITRM}}0^{\prime}_{\text{ITRM}}$. Then $x^{\prime}_{\text{ITRM}}\equiv_{\text{ITRM}}0^{\prime}_{\text{ITRM}}$.
\end{corollary}
\begin{proof}
We claim that, for $x<_{\text{ITRM}}0^{\prime}_{\text{ITRM}}$, we have $\omega_{\omega}^{\text{CK},x}=\omega_{\omega}^{\text{CK}}$: Clearly, $\omega_{\omega}^{\text{CK},x}\geq\omega_{\omega}^{\text{CK}}$.
Suppose the inequality was strict. As already $L_{\omega_{\omega}^{\text{CK}+1}}$ contains a real $c$ coding $L_{\omega_{\omega}^{\text{CK}}}$ by a standard fine-structural argument, such a code
is then also contained in $L_{\omega_{\omega}^{\text{CK},x}}[x]$ and thus ITRM-computable from $x$. But, as ITRM-programs in the empty oracle either halt inside $L_{\omega_{\omega}^{\text{CK}}}$ or do not halt at all,
we can compute $0^{\prime}_{\text{ITRM}}$ by evaluating truth-predicates for first-order formulas in $L_{\omega_{\omega}^{\text{CK}}}$, which can be done uniformly by an ITRM in the oracle $c$.
Hence $0^{\prime}_{\text{ITRM}}\leq_{\text{ITRM}}c\leq x$, which contradicts the assumption that $x$ lies strictly below $0^{\prime}_{\text{ITRM}}$.

We can now argue as in the proof of Lemma \ref{onlyjumprecog} to see that $x^{\prime}_{\text{ITRM}}\leq_{\text{ITRM}}0^{\prime}_{\text{ITRM}}$: To compute $x^{\prime}$ from $0^{\prime}_{\text{ITRM}}$, first compute $x$ from $0^{\prime}_{\text{ITRM}}$,
which is possible by assumption. Therefore, we have $x\in L_{\omega_{\omega}^{\text{CK},0^{\prime}_{\text{ITRM}}}}[0^{\prime}_{\text{ITRM}}]$ and as $\omega_{\omega}^{\text{CK},0^{\prime}_{\text{ITRM}}}>\omega_{\omega}^{\text{CK}}$,
it follows that $L_{\omega_{\omega}^{\text{CK},0^{\prime}_{\text{ITRM}}}}[0^{\prime}_{\text{ITRM}}]$ contains a code $\bar{c}$ for $L_{\omega_{\omega}^{\text{CK}}}[x]=L_{\omega_{\omega}^{\text{CK},x}}[x]$. Hence such a code
is ITRM-computable from $0^{\prime}_{\text{ITRM}}$. As above, $c^{\prime}$ can then be used to solve the halting problem relative to $x$, so $x^{\prime}_{\text{ITRM}}\leq_{\text{ITRM}}\bar{c}\leq_{\text{ITRM}}0^{\prime}_{\text{ITRM}}$, as desired.
\end{proof}

\textbf{Remark}: Shoenfield's jump inversion theorem implies that, for Turing machines, there is some $c<_{T}0^{\prime}$ such that $c^{\prime}=0^{\prime\prime}$. 
Corollary \ref{nojumpinv} shows that this is impossible for ITRMs.\\

We now characterize the computational strength of the recognizable jump.

\begin{defini}
For $N\models\text{ZFC}$, let $\mathcal{T}^{N}$ denote the set of parameter-free $\Sigma_{1}$-statements that hold in $N$.
\end{defini}

Note that $\mathcal{T}$ is absolute between transitive models of ZFC by Shoenfield's absoluteness theorem; as we are only interested in transitive models here, we can simply write $\mathcal{T}$.
In particular, we have $\mathcal{T}=\mathcal{T}^{L}$.

\begin{thm}{\label{recjumpsigma1}}
Let $0^{(r)}$ denote the recognizable jump for any of ITRMs, ITTMs and parameter-free OTMs. Then $\mathcal{T}\leq_{T}0^{(r)}$, i.e. the parameter-free $\Sigma_{1}$-theory of $L$ (and hence of $V$) $\mathcal{T}$ is Turing-reducible to
the recognizable jump.
\end{thm}
\begin{proof}
 Let $(\phi_{i}:i\in\omega)$ enumerate the set-theoretical $\Sigma_{1}$-statements without free variables in some natural way. Let $i\in\omega$ be arbitrary and assume that $L\models\phi_{i}$. 
Then there is a minimal $\alpha\in\text{On}$ such that $L_{\alpha}\models\phi_{i}$. By an easy condensation argument, we have $\alpha<\omega_{1}^{L}$.
Hence, there is some $<_{L}$-minimal real $r_{i}$ coding $L_{\alpha}$.

Now, given $i\in\omega$, it is easy to write a program $C_{i}$ that checks whether its oracle $y$ is a $<_{L}$-minimal code for an $L$-structure $L_{\alpha}$ such that
$L_{\alpha}\models\phi_{i}$: In fact, $C_{i}$ can be obtained from $i$ primitive recursively. Let $c(i)$ denote the index of $C_{i}$ in the enumeration $(P_{i}:i\in\omega)$
of programs.

But if $\phi_{i}$ holds, then there is a unique oracle $x$ such that $C_{i}^{x}\downarrow=1$, namely $x=r_{i}$. If, on the other hand $\phi_{i}$ is false, then there is no such oracle.
Hence $C_{i}$ recognizes a real number if and only if $\phi_{i}$ holds, i.e. $\phi_{i}$ holds if and only if $c(i)\in 0^{(r)}$. As $c(i)$ is primitive recursive in $i$, we can
Turing-compute $\Sigma_{1}$-truth from $0^{(r)}$.
\end{proof}

For ITTMs and ITRMs, we also get the converse. This can be seen as the recognizable counterpart of a theorem of Welch (see Corollary $1$ of \cite{We2}) showing that the (computational) ITTM-jump of a real $x$ corresponds
to a Master code for $L_{\lambda^{x}}[x]$, i.e. a $\Sigma_{1}$-truth predicate for that structure:

\begin{thm}{\label{sigma1recjump}}
 Let $0^{(r)}$ denote the recognizable jump for ITRMs or ITTMs. Then $0^{(r)}\leq_{T}\mathcal{T}$, i.e. the recognizable jump of $0$ for ITRMs and ITTMs is Turing-reducible (and hence, by Theorem \ref{recjumpsigma1}, Turing-equivalent) 
to the parameter-free $\Sigma_{1}$-theory of $L$.
\end{thm}
\begin{proof}
We start by noting that, for $P$ an ITRM- or an ITTM-program, the following statements are $\Sigma_{1}$:
\begin{enumerate}
 \item There is some $x\subseteq\omega$ such that $P^{x}\downarrow=1$
 \item There is some $x\subseteq\omega$ such that $P^{x}\uparrow$ or $P^{x}\downarrow\notin\{0,1\}$
 \item There are $x,y\subseteq\omega$ such that $x\neq y$ and $P^{x}\downarrow=1$ and $P^{y}\downarrow=1$ 
\end{enumerate}
This is straightforward for (1), which can be expressed as `There is $x\subseteq\omega$ and a $P$-computation in the oracle $x$ that contains a halting state with $1$ written in the first register/on the first tape cell,
and similarly for (3) and the claim that there is $x$ such that $P^{x}\downarrow\notin\{0,1\}$. The only complication arises with the statement that $P^{x}$ diverges for some $x$. We treat the ITRM- and the ITTM-case separately:

For ITRMs, $P^{x}\uparrow$ means, by Theorem \ref{hp}, that $L_{\omega_{\omega}^{\text{CK},x}}[x]\models P^{x}\uparrow$, which can be expressed as `There are a set $y$ and an ordinal $\alpha$ such that
$y=L_{\alpha}[x]$, $y$ contains infinitely many $x$-admissible ordinals and $y\models P^{x}\uparrow$'. As `$y=L_{\alpha}[x]$' is $\Sigma_{1}$ in $x$ and $\alpha$, this is a $\Sigma_{1}$-statement.

For ITTMs, $P^{x}\uparrow$ means, by definition of $\lambda^{x}$, that $L_{\lambda^{x}}[x]\models P^{x}\uparrow$; this can, by \cite{We3}, be expressed as `There are sets $a,b,c$ and ordinals $\alpha,\beta,\gamma$ such that
$a=L_{\alpha}[x]$, $b=L_{\beta}[x]$ and $c=L_{\gamma}[x]$, $a\prec_{\Sigma_{1}}b$, $b\prec_{\Sigma_{2}}c$ and $a\models P^{x}\uparrow$', which again is $\Sigma_{1}$.

Now, it is easy to see that the statements $\phi_{1}^{P,\text{ITRM}},\phi_{2}^{P,\text{ITRM}},\phi_{3}^{P,\text{ITRM}}$ and $\phi_{1}^{P,\text{ITTM}},\phi_{2}^{P,\text{ITTM}},\phi_{3}^{P,\text{ITTM}}$
 expressing (1)-(3) for ITRMs and ITTMs in $\Sigma_{1}$-form, respectively, can be obtained primitive recursively from $P$. Note that, as parameter-free $\Sigma_{1}$-statements, all of these statements
are absolute between $V$ and $L$ by Shoenfield absoluteness. Hence $P$ is recognizing if and only if $\phi_{1}^{P}$ holds and $\phi_{2}^{P},\phi_{3}^{P}$ are false (for the machine type in question), 
which, by absoluteness, holds if and only if the hold in $L$, which in turn can be checked by using
the (oracle coding the) parameter-free $\Sigma_{1}$-theory of $L$.

\end{proof}

\textbf{Remark}: This doesn't work for OTMs since there is no bound on the halting time of an OTM in the oracle $x$ that is $\Sigma_{1}$-definable in the oracle $x$, so that $\exists{x}P^{x}\uparrow$ is not $\Sigma_{1}$-expressible.
It is in fact easy to see that it is not: For there is a parameter-free OTM that, given a parameter-free $\Sigma_{1}$-statement $\phi$, halts if and only if $\phi$ holds: $P$ works simply by writing $L$ on the tape and checking
at each level whether $\phi$ holds in it. Also, the statement that an OTM-program $Q$ halts is clearly $\Sigma_{1}$-expressible. If the divergence of $Q$ was $\Sigma_{1}$-expressible as well, OTMs could solve their own halting problem,
which they clearly can't. At this moment, we do not know of a characterization of the recognizable OTM-jump in the spirit of Theorem \ref{sigma1recjump}.\\

The concept of relativized recognizability makes it tempting to define `degrees of recognizability'. This, however, is hindered by the observation made in \cite{Ca1} that relativized recognizability is not transitive. There are two ways
around this: One can either give up on having degrees and merely study the reducibility relation on single real numbers, or one can replace relativized recognizability with its transitive closure to make it an equivalence relation. We shall take the second route here.

\begin{defini}{\label{recclos}}
We let $x\preceq_{M}y$ if and only if there is a finite sequence $x=z_{0},z_{1},...,z_{n}=y$ such that $z_{i}\leq^{r}_{M}z_{i+1}$ for $i\in\{0,...,n-1\}$. If $x\preceq_{M}y$, we say that
$x$ is heriditarily $M$-recognizable from $y$.
\end{defini}

\textbf{Remark}: In fact, the reduction turns out to be much simpler: A two-step iteration suffices to give the whole transitive closure. 
Even more: It is not hard to show (see e.g. \cite{CS1}) that $x\preceq_{M}y$ if and only if there is $z$ such that $z\leq^{r}_{M}y$ and $x$ is primitive recursive in $z$ (and less).

\begin{defini}{\label{receq}}
 $x$ and $y$ are called $M$-recognizably equivalent, written $x\asymp_{M} y$ if and only if $x\preceq_{M}y$ and $y\preceq_{M}x$. It is easy to see that $M$-recognizable equivalence is indeed an equivalence relation. $[x]^{r}_{M}$ will denote the 
$\asymp_{M}$-equivalence class of $x$, called $M$-recognizability degree of $x$.
\end{defini}

We can now formulate and solve a recognizable analogue for Post's problem. For convenience, we write $\alpha_{M}$ where $\alpha_{\text{ITRM}}^{x}=\omega_{\omega}^{\text{CK},x}$, $\alpha_{\text{ITTM}}^{x}=\lambda^{x}$ and
$\alpha_{\text{OTM}}=\sigma^{x}$ for $x\subseteq\omega$. We call $\alpha_{M}^{x}$ the $M$-characteristic ordinal for $x$.

\begin{thm}{\label{intermedrecogdeg}}
For $M\in\{\text{ITRM},\text{ITTM}\}$, there is $x$ such that $[0]^{r}_{M}\prec_{M}[x]^{r}_{M}\prec_{M}[0^{r}_{M}]^{r}_{M}$.
\end{thm}
\begin{proof}
Let $x$ be the $<_{L}$-minimal Cohen-generic real over $L_{\sigma+1}$. We claim that $x\notin [0]^{r}_{M}$: For if it was, then $x$ would, by our remark above, be recursive in a recognizable real,
but all recognizable real numbers and all real numbers recursive in recognizable real numbers are contained in $L_{\sigma}$, which $x$ clearly is not. 

Second, we need to see that $0^{r}_{M}\notin [x]^{r}_{M}$. We claim that $\alpha^{0^{r}_{M}}_{M}>\sigma=\sigma^{x}\geq\alpha^{y}_{M}$ for every $y\in [x]^{r}_{M}$, which suffices. 

For the first inequality, note that a real number coding a well-ordering of order-type $\sigma$ is $M$-computable from $0^{r}_{M}$. To see this, first split $\omega$ into $\omega$ many disjoint infinite portions $(A_{i})_{i\in\omega}$
in some effective way and declare any element of $A_{i}$ to be smaller than any element of $A_{j}$ for $i<j$. The $A_{i}$ are then ordered as follows:
Using $0^{r}_{M}$, search for the $i$th $M$-program $Q$ such that $Q$ recognizes a code for a well-ordering. `There is some $x$ that codes a well-ordering such that $Q^{x}\downarrow=1$' is a $\Sigma_1$-statement
whose truth value can be evaluated using $0^{r}_{M}$ by Theorem \ref{recjumpsigma1}. By the same argument, every bit of the real number $c_{Q}$ recognized by $Q$ can be computed from $0^{r}_{M}$. Now order $A_{i}$
in the way coded by $c_{Q}$. The resulting real number will code a well-ordering whose order-type is the ordinal sum of all ordinals with a recognizable code, which is $\sigma$.

That $\sigma^{x}=\sigma$ follows from the genericity of $x$: Let $P$ be an OTM-program such that $P^{x}$ halts in $\alpha$ many steps. By the forcing theorem for admissible sets (see e.g. Lemma 32 of \cite{CS}),
there is a condition $p\subseteq x$ such that $p\Vdash`P^{x}$ halts in $\alpha$ many steps'. Hence $L\models\exists{\alpha}(p\Vdash`P^{x}$ halts in $\alpha$ many steps'), which is a $\Sigma_{1}$-statement and hence
must become true already in $L_{\sigma}$. Thus, we must have $\alpha<\sigma$ and $\sigma$ is the supremum of OTM-halting times in the oracle $x$, so $\sigma^{x}=\sigma$.

For the second inequality, suppose $z$ is such that $z\leq^{r}_{M}x$ and $y$ is primitive recursive in $z$. Clearly, we must have $\alpha^{y}_{M}\leq\alpha^{z}_{M}$, so it suffices to see that
$\alpha^{z}_{M}\leq\sigma$. Let $P$ be a program that recognizes $z$ relative to $x$. Then $\exists{a}P^{a\oplus x}\downarrow=1$ is a $\Sigma_{1}$-statement in the parameter $x$ which, for $x\in L$, is absolute between $V$ and $L$ by
Shoenfield absoluteness and moreover, by definition of $\sigma^{x}$, either becomes true in some $L_{\alpha}[x]$ with $\alpha<\sigma^{x}$ or is false. As it is true by assumption, we must have $z\in L_{\sigma^{x}}[x]=L_{\sigma}[x]$. Now
let $Q$ be an $M$-program such that $Q^{z}$ halts. Then $\exists{a}\exists{\alpha}[P^{a\oplus x}\downarrow=1\wedge (P^{a}\downarrow)^{L_{\alpha}[a]}]$ is again a $\Sigma_{1}$-statement in the parameter $x$, which is true by assumption
and must hence hold in some $L_{\beta}[a]$ with $\beta<\sigma^{x}$. Hence the halting times of $M$-programs in the oracle $z$ are indeed majorized by $\sigma^{x}=\sigma$, as desired.

Finally, we show that $x\leq^{r}_{M} 0^{r}_{M}$. This, together with the preceeding, shows that the strict inequality holds for such an $x$ and hence that $x$ is as desired.
But as we saw above that $\sigma<\alpha_{M}^{0^{r}_{M}}$, we have $L_{\sigma+\omega}\subseteq L_{\alpha_{M}^{0^{r}_{M}}}[0^{r}_{M}]$. Now, the $\Sigma_{1}$-Skolem hull of $\emptyset$ in $L_{\sigma}$ will
be an $L$-level reflecting every $\Sigma_{1}$-statement that holds in $L_{\sigma}$ and hence, by definition of $\sigma$, be isomorphic to $L_{\sigma}$. Thus, a surjection of $\omega$ onto $L_{\sigma}$ is definable
over $L_{\sigma}$ and hence $L_{\sigma+1}$ is countable in $L_{\sigma+2}$. Thus, a real number $y$ that is Cohen-generic over $L_{\sigma+1}$ is contained in $L_{\sigma+3}\subseteq L_{\alpha_{M}^{0^{r}_{M}}}[0^{r}_{M}]$,
so the $<_{L}$-smallest such $y$ is in particular $M$-computable from $0^{r}_{M}$ and hence $M$-recognizable from $0^{r}_{M}$.


\end{proof}

We conclude by observing that the structure of $M$-recognizability degrees turns out to depend heavily on the set-theoretical background:

\begin{thm}{\label{recdeginL}}
Let $M\in\{\text{ITRM},\text{ITTM},\text{OTM}\}$. Assume that $V=L$. Then the $M$-recognizability degrees are linearly ordered by the ordering induced by the canonical well-ordering $<_{L}$ of $L$.
\end{thm}
\begin{proof}
Relative to $x\subseteq$, we can recognize the $<_{L}$-minimal code of the minimal $L$-level $L_{\alpha_{x}}$ containing $x$. From a code for $L_{\alpha_{x}}$, every real $z\in L_{\alpha_{x}}$
is computable and hence recognizable. Hence every real in $L_{\alpha_{x}}$ is heriditarily recognizable from $x$.

Now, for any $x,y\in L$, we have either $\alpha_{x}\leq\alpha_{y}$ or $\alpha_{x}>\alpha_{y}$, i.e. $x\in L_{\alpha_{y}}$ or $y\in L_{\alpha_{x}}$.
Hence one of them is heriditarily $M$-recognizable from the other. Moreover, if $x\leq_{L}y$, then $\alpha_{x}\leq\alpha_{y}$, so $x$ is heriditarily $M$-recognizable from $y$.
\end{proof}

On the other hand:

\begin{prop}{\label{recdeggenext}}
Let $M\in\{\text{ITRM},\text{ITTM},\text{OTM}\}$. If $x,y$ are mutually Cohen-generic over $L$, then neither $x\preceq_{M}y$ nor $y\preceq_{M}x$.
\end{prop}
\begin{proof}
Assume for a contradiction that $x\preceq_{M}y$. Let $P$ be a program that recognizes a real $z$ relative to $y$ such that $x$ is primitive recursive in $x$ which exists by the remark following Theorem \ref{recclos}.
Then $x$ is $\Sigma_1$-definable from $y$ and hence, by Shoenfield absoluteness, we get $x\in L[y]$, which contradicts the assumption that $x$ is Cohen-generic over $L[y]$.
\end{proof}

Consequently, the existence of $\preceq_{M}$-incomparable $M$-degrees of recognizability is independent of ZFC for $M\in\{\text{ITRM},\text{ITTM},\text{OTM}\}$. The study of the structure of
the degrees of recognizability will hence need to work under set-theoretical extra assumptions to obtain substantial results.

\section{Conclusion and Further Work}
It is natural to ask what happens when we replace the condition of non-meagerness by the condition of positive Lebesgue measure in results like \ref{ITRMcomeager} or \ref{ITTMcomeager}.
Hence, we ask:

\medskip

\textbf{Question}: Suppose that $Y\subseteq[0,1]$ is a set of positive Lebesgue measure, and let $x\subseteq\omega$ be such that $x$ is ITRM-recognizable (uniformly or not) relative to every
$y\in Y$. Does it follow that $x$ is ITRM-recognizable?

\medskip

\textbf{Question}: Same question for ITTM-recognizability instead of ITRM-recognizability.

This and other related topics can be dealt with using random forcing
over models of $KP$, which will be treated in future work with Philipp Schlicht. 

Another possible topic to pursue is relativized recognizability for $\alpha$-Turing machines and $\alpha$-register machines, both of the resetting
and the unresetting type. Introductions to and various results about computational strength, computability from random oracles as well as on recognizability for these machines can be found in 
\cite{KS}, \cite{Rin}, \cite{Ca} and \cite{CS}. In particular, we ask:

\medskip

\textbf{Question}: For which $\alpha$ is it true that the recognizability by an $\alpha$-Turing machine or an $\alpha$-register machine of a real number $x$ relative to all elements of a set $Y$ of real numbers that is `large' 
(e.g. in the sense of being comeager or of Lebesgue measure $1$)?

\medskip

\textbf{Question}: What is the strength of the recognizable jump of $0$ for $\alpha$-Turing machines and $\alpha$-register machines? What does the recognizable degree structure for such machines look like?

\medskip

\textbf{Question}: Characterize the recognizable jump of $0$ for Ordinal Turing Machines without parameters, possibily analogous to Theorem \ref{sigma1recjump}. In particular, is it provable in ZFC that $0^{r}_{\text{OTM}}$
is constructible?

\section{Acknowledgements}

We thank Philipp Schlicht for a discussion on Corollary \ref{MAnonuniform} (and the remark following it) as well as various helpful comments on the presentation of the proof of Theorem \ref{ITRMcomeager}
and discussions during which Lemma \ref{whomgennonrec} and Theorem \ref{laver} were suggested, along with parts of their proofs. We also thank the two anonymous referees for suggesting various 
improvements on our presentation.

\section{Appendix: The relativized Jensen-Karp-Theorem}

In this section, we prove a relativization and a slight strengthening of the Jensen-Karp theorem (see section $5$ of \cite{JK}). The theorem says that, when $\phi$ is a $\Sigma_{1}$-statement with real parameters in $L_{\alpha}$ and $\alpha$
is a limit of admissible ordinals, then $V_{\alpha}\models\phi$ if and only if $L_{\alpha}\models\phi$. In the proof of our theorem concerning ITTMs above, we made use of the following relativization:

\begin{thm}{\label{parameterJK}}
Let $x\subseteq\omega$, $\phi$ a $\Sigma_{1}$-formula in the parameter $x$, possibly with further real parameters in $L_{\omega_{\omega}^{\text{CK},x}}[x]$.
Then $\phi$ is absolute between $V_{\omega_{\omega}^{\text{CK},x}}$ and $L_{\omega_{\omega}^{\text{CK},x}}[x]$.
\end{thm}

Moroever, we have the following strengthening:

\begin{thm}{\label{strongJK}}
Let $\phi$ be a $\Sigma_{1}$-formula, possibly with real parameters $z$ in $L_{\alpha}$. Let $\beta^{+}_{z}$ denote the smallest $z$-admissible ordinal $>\beta$.
Then $\phi$ is absolute between $V_{\alpha^{++}_{z}}$ and $L_{\alpha^{++}_{z}}$.
\end{thm}
\begin{proof}
 (Sketch) This will follow from the construction below as the primitive-recursive definition of the transitive $\in$-model $(t,\in)$ of $\{\phi\}$ defined there needs only the next admissible
$\alpha^{+}_{z}$ as a parameter and can hence be carried out in $L_{\alpha^{++}_{z}}$; so $L_{\alpha^{++}_{z}}$ contains a transitive model of $\phi$, and by upwards absoluteness of $\Sigma_{1}$-formulas,
$L_{\alpha^{++}_{z}}\models\phi$.
\end{proof}

\begin{corollary}{\label{recognizablespositions}}
Let $x$ be ITRM-recognizable by an ITRM-program $P$ using $n$ registers. Then $x\in L_{\omega_{n+2}^{\text{CK},x}}$.
\end{corollary}
\begin{proof}
Since $P$ recognizes $x$, $P^{y}$ will halt for every oracle $y$. As $P$ uses only $n$ registers, we have that for every $y\subseteq\omega$, $P^{y}$ will run for less than $\omega_{n+1}^{\text{CK},x}$ many steps.
Hence the statement $P^{y}\downarrow=1$ is absolute between $L_{\omega_{n+1}^{\text{CK},y}}$ and $V_{\omega_{n+1}^{\text{CK},y}}$, provided that $y\in L_{\omega_{n+1}^{\text{CK},y}}$.
Let $\tau$ be the halting time of $P^{x}$, so $\tau<\omega_{n+1}^{\text{CK},x}$. By Theorem \ref{strongJK}, the parameter-free $\Sigma_{1}$-statement $\psi\equiv$`There is $y\subseteq\omega$ such that $P^{y}\downarrow=1$'
is absolute between $L_{\tau^{++}_{x}}=L_{\omega_{n+2}^{\text{CK},x}}$ and $V_{\tau^{++}_{x}}=V_{\omega_{n+2}^{\text{CK},x}}$. As $P$ recognizes $x$ and $x\in V_{\omega+1}\subseteq V_{\omega_{n+2}^{\text{CK},x}}$, by the bound on the halting times 
(see Theorem \ref{hp})
and absoluteness of computations, we have $V_{\omega_{n+2}^{\text{CK},x}}\models\psi$, so $L_{\omega_{n+2}^{\text{CK},x}}\models\psi$, so $x\in L_{\omega_{n+2}^{\text{CK},x}}$ by definition of recognizability and absoluteness of computations.
\end{proof}

Both theorems can be combined in the obvious way. Though we use ITRMs instead of Jensen and Karp's primitiv recursive set functions, the proofs work very much along the lines of Jensen's and Karp's proof of the second Theorem in section $5$ of \cite{JK};
it is for the sake of completeness that we give the details of the proof of Theorem \ref{parameterJK} here.

We now prove Theorem \ref{parameterJK}.

\begin{proof}
                                                                                                                                                                                                                                                                                                                                                                                                                                                                                                                                           
We show that, if $x\subseteq\omega$ and $\phi(x)\equiv\exists{y}\psi(x,y)$ with $\psi$ a $\Delta_{0}$-formula is such that $V_{\omega_{\omega}^{\text{CK},x}}\models\phi(x)$, then $L_{\omega_{\omega}^{\text{CK,x}}}[x]\models\phi(x)$.
The construction we use is due to Jensen and Karp (\cite{JK}).

The idea is to compute with an ITRM, using the oracle $x$, a (code for a) transitive set $M\ni x$ such that $M\models\phi(x)$. To this end, we extend $\phi(x)$ to an appropriate complete theory in which every
existential statement is witnessed by a constant; then, we use the Henkin model construction, which can be effectivized on an ITRM. 

Once this is done, the claim follows: As $c$ is ITRM-computable from $x$, we have $c\in L_{\omega_{\omega}^{\text{CK},x}}[x]$. It is then not hard to conclude that $M\in L_{\omega_{\omega}^{\text{CK},x}}[x]$.
As $M\models\phi(x)$, there is some $y\in M$ such that $M\models\psi(x,y)$. As $M$ is transitive, $\psi(x,y)$ holds in the real world. Finally, as $y\in M\in L_{\omega_{\omega}^{\text{CK},x}}[x]$ and 
the latter is transitive, we have $y\in L_{\omega_{\omega}^{\text{CK},x}}[x]$, so $L_{\omega_{\omega}^{\text{CK},x}}[x]\models\phi(x)$, as desired.

We could of course attempt to compute a Henkin-model for $\phi(x)$ right away on an ITRM in the oracle $x$. The problem is that this model might fail to be transitive (or to be isomorphic to a transitive model),
so that the witness it contains for $\phi(x)$ may fail to witness this statement in $V$. Thus -- and this is the crucial point of the Jensen-Karp-proof -- it has to be ensured that the model will be well-founded,
and thus isomorphic to a transitive model. This is the point of the following construction.

As $L_{\omega_{\omega}^{\text{CK},x}}[x]\models\phi(x)$, there is $\beta<\omega_{\omega}^{\text{CK},x}$ such that $L_{\beta}[x]\models\phi(x)$; we pick such a $\beta$, say the minimal one.
Then a code $r_{\beta}$ for $\beta$ is ITRM-computable from $x$. We may thus use $\beta$ freely in the following construction.

We introduce constant symbols $(c_{\iota}:\iota<\beta)$ and another constant symbol $c_{x}$ and form the theory $\mathcal{T}$ using the binary relation symbol $\in$ and the unary function symbol $\text{rk}$ (intended to denote the rank function)
which will contain the following statements:
$\phi(c_{x})$, $c_{i}\in c_{x}$ for $i\in x$, $c_{i}\notin c_{x}$ for $i\notin x$, $c_{\iota}\in c_{\iota^{\prime}}$ for $\iota<\iota^{\prime}\leq\beta$, $c_{\iota}\notin c_{\iota^{\prime}}$ for $\iota^{\prime}\leq\iota\leq\beta$
and $\text{On}(c_{\iota})$ for $\iota\leq\beta$. Moreover, $\mathcal{T}$ contains the axiom of extensionality and statements giving the relevant information about the rank function, namely $\forall{x,y}(x\in y\rightarrow\text{rk}(x)\in\text{rk}(y))$,
$\forall{x}\text{rk}(x)\in\text{On}$ and $\forall{x\in\text{On}}(\text{rk}(x)=x)$.

Clearly, this theory $\mathcal{T}$ is ITRM-computable from $x$; moreover, $\mathcal{T}$ is consistent, as $V_{\beta+1}$ with the standard interpretations of $\in$ and $\text{rk}$ is a model of $\mathcal{T}$.

We now extend, in an ITRM-effective manner, $\mathcal{T}$ to a suitable complete consistent theory $\mathcal{T}^{\prime}$ in which every existential statement is witnessed by one of $\omega$ many new constants $\{d_{i}:i\in\omega\}$,
and which moreover has the property that, if $\mathcal{T}^{\prime}$ contains the statement $\text{On}(a)$ for some $a$, then it also contains the statement $a=c_{\iota}$ for some $\iota\leq\beta$. This latter condition
ensures that the model we obtain in the end will only have `actual' ordinals and thus be well-founded.

Let $\mathcal{L}$ be the first-order language using the function, relation and constant symbols just described. 
We enumerate the $\mathcal{L}$-formulas with one free variable in the order type $\omega$ as $(\psi_{i}:i\in\omega)$.
We now build a tree $\mathcal{B}$ on the set of the closed $\mathcal{L}$-formulas. For all $i\in\omega$, at the $3i$th level, every node of $\mathcal{B}$ will have $\omega$ many immediate successors labelled
$\exists{z}\psi_{i}(z)\rightarrow\psi_{i}(d_{j})$, $j\in\omega$. At the $(3i+1)$th level, we will again have $\omega$ many immediate successors for each node, labelled $\text{On}(d_{i})\rightarrow (d_{i}=c_{\iota})$, $\iota\leq\beta$;
these can be arranged in order type $\omega$ as we have a real number coding $\beta$ available. Finally, at level $3i+2$, we have the $i$th closed formula of $\mathcal{L}$. A finite sequence $\vec{s}$ of formulas will
enter $\mathcal{B}$ if and only if $\mathcal{T}\cup s$ is consistent, which can easily be tested with an ITRM. Thus, $\mathcal{B}$ is ITRM-computable from $x$.

Any infinite branch of $\mathcal{B}$ will be a complete consistent extension of $\mathcal{T}$. We know that such an extension exists, as we can form the elementary hull of $\beta$ in $V_{\beta+1}$, which will result in a countable structure 
in which $d_{i}$ can be interpreted as the $i$th element and all other symbols can be interpreted in the obvious way. To find such a branch with an ITRM given $\mathcal{B}$, we use the
algorithm used on ITRMs to test for well-foundedness of binary relations: $\mathcal{B}$ has an infinite branch if and only if the binary relation defined by $\eta<\eta^{\prime}$ if and only if $\eta$ is a successor of $\eta^{\prime}$ in $\mathcal{B}$
is not well-founded. Moreover, this algorithm can also be used to test whether a given sequence of formulas extends to an infinite branch of $\mathcal{B}$. In this way, always selecting the lexically minimal (leftmost)
continuation of a branch that still extends to an infinite branch of $\mathcal{B}$, we can ITRM-compute an infinite branch of $\mathcal{B}$ from $x$.

Now let $\mathcal{T}^{\prime}$ be a complete and consistent extension of $\mathcal{T}$ as described and consider the model $M$ built from the constant symbols of $\mathcal{L}$ in the usual way: We form equivalence classes
of constants depending on whether $\mathcal{T}^{\prime}$ contains the statement that they are equal and, for two such classes $[d],[d^{\prime}]$, we let $\text{rk}([d])=[d^{\prime}]$ if and only if $\text{rk}(d)=d^{\prime}\in\mathcal{T}^{\prime}$
and $[d]\in[d^{\prime}]$ if and only if $(d\in d^{\prime})$ is contained in $\mathcal{T}^{\prime}$. Clearly, the resulting structure $S$ has no infinite descending $\in$-sequences, as these would translate into an infinite descending $\in$-sequence
among the $c_{\iota}$, which would in turn induce an infinite descending sequence in their indices, i.e. in the ordinals, a contradiction. Moreover, taking the transitive collapse $M$ of $S$, $c_{\iota}$ will be mapped to $\iota$ for all $\iota\leq\beta$,
so in particular, $c_{i}$ will be mapped to $i$ for $i\in\omega$ and thus, by extensionality in $S$, $c_{x}$ will be mapped to $x$. As $S$ satisfies $\phi(x)$ by definition, the same holds for $M$. Thus $M$ is as desired.


It is now possible to ITRM-compute a code for $M$ from $\mathcal{T}^{\prime}$ by implementing the above construction on an ITRM, which is easily possible.

\end{proof}

\end{document}